\documentclass[12pt]{article}
\usepackage{amsthm}
\usepackage{amsmath}
\usepackage{amsfonts}
\usepackage{amssymb}
\usepackage[mathscr]{eucal}
\usepackage{enumerate}
\usepackage{graphicx}
\usepackage{pgf,tikz}
\usetikzlibrary{arrows,backgrounds,patterns,shapes.geometric,calc,positioning}
\usepackage{color}
\usepackage{subfig}

\DeclareMathOperator{\supp}{supp}

\DeclareMathOperator{\diam}{diam}

\usepackage{accents}
\usepackage[pdftex,hidelinks]{hyperref}

\newcommand{\ri}{{\mathrm{i}}}
\newcommand{\re}{{\mathrm{e}}}
\newcommand{\rd}{\mathrm{d}}
\newcommand{\hull}{\mathrm{Hull}}

\newcommand{\R}{\mathbb{R}}

\newcommand{\N}{\mathbb{N}}

\newcommand{\C}{\mathbb{C}}

\newcommand{\cS}{\mathcal{S}}
\newcommand{\cD}{\mathcal{D}}
\newcommand{\cH}{\mathcal{H}}
\newcommand{\cI}{\mathcal{I}}

\newcommand{\cP}{\mathcal{P}}
\newcommand{\cT}{\mathcal{T}}

\newcommand{\cW}{{\mathcal W}}

\newcommand{\bA}{\mathbf{A}}
\newcommand{\bB}{\mathbf{B}}

\newcommand{\normt}[2]{\|#1\|_{#2}}

\newtheorem{thm}{Theorem}[section]
\newtheorem{lem}[thm]{Lemma}
\newtheorem{defn}[thm]{Definition}
\newtheorem{prop}[thm]{Proposition}

\newtheorem{rem}[thm]{Remark}

\newcommand{\tH}{\widetilde{H}}
\newcommand{\dimH}{{\rm dim_H}}

\begin{document}


%
%

\title{Acoustic scattering by impedance screens/cracks with fractal boundary: well-posedness analysis and boundary element approximation 
}
\date{}
\author{%
	{\sc
		J.\ Bannister\thanks{Email: joshua.smith.14@ucl.ac.uk},  
		A.\ Gibbs\thanks{Email: andrew.gibbs@ucl.ac.uk},
		D.\ P.\ Hewett\thanks{Email: d.hewett@ucl.ac.uk (corresponding author)}
	} \\[2pt]
	Department of Mathematics, University College London, UK.
}
\maketitle

\begin{abstract}
We study time-harmonic scattering in $\R^n$ ($n=2,3$) by a planar screen (a ``crack'' in the context of linear elasticity), assumed to be a non-empty bounded relatively open subset $\Gamma$ of the hyperplane $\Gamma_\infty=\R^{n-1}\times \{0\}$, on which impedance (Robin) boundary conditions are imposed. In contrast to previous studies, $\Gamma$ can have arbitrarily rough (possibly fractal) boundary. 
To obtain well-posedness for such $\Gamma$ we show how the standard impedance boundary value problem and its associated system of boundary integral equations must be supplemented with additional solution regularity conditions, which hold automatically when $\partial\Gamma$ is smooth. We show that the associated system of boundary integral operators is compactly perturbed coercive in an appropriate function space setting, strengthening previous results. 
This permits the use of Mosco convergence to prove convergence of boundary element approximations on smoother ``prefractal'' screens to the limiting solution on a fractal screen. 
We present accompanying numerical results, validating our theoretical convergence results, for three-dimensional scattering by a Koch snowflake and a square snowflake. 
\end{abstract}



\section{Introduction}

We consider time-harmonic acoustic scattering in $\R^n$ ($n=2$ or $3$) by a non-empty, bounded, relatively open subset $\Gamma$ of the hyperplane $\Gamma_\infty:=\R^{n-1}\times \{0\}\subset \R^n$. Following the convention in acoustics and electromagnetism we refer to $\Gamma$ as a ``screen'', but by Babinet's principle $\Gamma$ could equivalently represent a finite ``aperture'' in a screen of infinite extent, and in linear elasticity $\Gamma$ would be referred to as a ``crack'' or ``fracture''.  
The classical case where the relative boundary $\partial\Gamma$ of $\Gamma$ is smooth (e.g.\ Lipschitz) has been studied extensively, both analytically and numerically (for a list of references relating to Dirichlet and Neumann problems see \cite{chandler-wilde2019}). The current paper forms part of an ongoing research programme dedicated to extending the classical theory of screens to the case where $\partial\Gamma$ is non-smooth, for example fractal. A prototypical example for the current paper in the case $n=3$ would be where the screen $\Gamma\subset\R^2\times\{0\}\subset\R^3$ is the interior of the Koch snowflake (see Figure \ref{fig:Snowflakes}, top row).

Our motivation for considering scattering by fractal structures is twofold. 
On the one hand, from a modelling perspective, fractals provide a natural mathematical model for the multiscale roughness of many natural and man-made scatterers, such as dendritic structures like the human lung \cite{achdou2007transparent}, snowflakes, ice crystals and other atmospheric particles \cite{StWe:15}, and fractal antennas/transducers in electromagnetics/ultrasonics (\cite{GhSiKa:14,algehyne2018analysis}).
On the other hand, from a mathematical analysis perspective, the study of PDEs and integral equations on domains with fractal boundaries 
raises challenging questions that are stimulating new fundamental research in the theory of function spaces, variational problems and numerical discretizations (see, e.g.,\ \cite{mosco2013analysis}). 


The current paper concerns scattering by screens on which an impedance (Robin) boundary condition of the form $\partial_n u + \lambda u=0$ is imposed. Here $\partial_n$ is a normal derivative and $\lambda$ is a complex-valued quantity, possibly depending on position along the boundary.  Assuming the normal direction points into the propagation domain, one typically assumes $\Im[\lambda]\geq 0$ to prohibit energy creation at the boundary. When $\Im[\lambda]>0$ the boundary absorbs some portion of the wave energy incident upon it, while when $\Re[\lambda]\neq0$ the boundary imposes a phase shift on the reflected wave. 
Boundary value problem (BVP) and boundary integral equation (BIE) formulations of the impedance screen (or crack) problem were presented 
in \cite{kress2003integral} and \cite{ben2013application} for screens in 2D and 3D with implicitly assumed, but unspecified, boundary smoothness. 
The objectives of the paper are to (i) clarify the smoothness assumptions under which the formulations in \cite{kress2003integral} and \cite{ben2013application} are well-posed; (ii) present generalised formulations which are well-posed for arbitrary $\Gamma$ (including with fractal boundary) and for more general impedance parameters than were considered in  \cite{kress2003integral,ben2013application}; (iii) prove rigorous results concerning the convergence of boundary element method (BEM) approximations on smoother ``prefractal'' screens to the limiting solution on a fractal screen in the joint limit of prefractal and mesh refinement. 

The paper builds strongly on the related papers \cite{ScreenPaper}, which derives well-posed BVP and BIE formulations for sound-soft (Dirichlet) and sound-hard (Neumann) acoustic scattering by arbitrary bounded screens $\Gamma\subset \Gamma_\infty$, and \cite{chandler-wilde2019}, which provides a rigorous convergence analysis for BEM approximations in the sound-soft case.
As in these papers, our analysis is based on a variational formulation of the BIEs in appropriate fractional Sobolev spaces, 
which allows the question of prefractal to fractal convergence to be rephrased in terms of the Mosco convergence (defined in \ref{sec:Mosco} below) of the discrete BEM subspaces to the Sobolev space in which the limiting solution (on the fractal screen) lives.
Hence, properties of fractional Sobolev spaces and integral operators on non-smooth sets, explored recently in \cite{ChaHewMoi:13,InterpolationCWHM,HewMoi:15,CoercScreen2,caetano2018}, are crucial to our arguments throughout. 
Our well-posedness analysis is more involved than that of \cite{ScreenPaper}, since, in contrast to the Dirichlet and Neumann cases, the system of boundary integral operators for the impedance problem is not known to be coercive (strongly elliptic). 
We note that preliminary results on the impedance well-posedness analysis were presented in the conference paper \cite{HewettBannisterWaves2019}.


Our analysis of the impedance BVP and BIE formulations follows closely the approach in \cite{ben2013application}, but differs from it in a number of respects. First, we take more care than in \cite{ben2013application} to distinguish between distributions on $\Gamma_\infty$ and distributions on $\Gamma$, and to state precisely in which function space setting the impedance boundary conditions are to be imposed. This is crucial when $\partial\Gamma$ is non-smooth since certain classical embedding results may fail (e.g.\ $\tilde{H}^{-1/2}(\Gamma)$ may fail to be embedded in $H^{-1/2}(\Gamma)$), meaning that we need to adopt a different function space setting for the BIE than was used in \cite{ben2013application} (see Remark \ref{rem:Comparison}). 
Second, having made this change of function space setting we are able to prove that the boundary integral operator is compactly perturbed coercive, rather than merely compactly perturbed invertible, as was shown in \cite{ben2013application}. This improvement is crucial when it comes to using Mosco convergence to prove prefractal to fractal convergence (see Theorem \ref{th:NA}). 
Third, the lack of smoothness of $\partial\Gamma$ allows the possibility that $\tH^{1/2}(\Gamma)\neq \{u\in H^{1/2}(\Gamma_\infty):\supp u\in \overline\Gamma\}$ and/or $H^{-1/2}_{\partial\Gamma}\neq \{0\}$, meaning that we need to impose additional smoothness assumptions on the BVP solution in order to guarantee BVP uniqueness (see Definition \ref{def:BVP} and Theorem \ref{Thm2}). 
Fourth, the proof of BVP uniqueness in \cite{ben2013application} cites the earlier paper \cite{kress2003integral}, where a proof is given only for the case where $\Gamma$ is a (curved) line segment, with the singular behaviour near $\partial\Gamma$ being dealt with by a cut-off argument. Rather than attempting to generalise this argument, we provide a different uniqueness proof based on distributional calculus that is valid in the general non-smooth case (see Theorem \ref{thm:BVPuniqueness}).
Fifth, we argue that while the well-posedness analysis in \cite{ben2013application} purports to hold for general bounded impedance functions $\lambda^\pm\in L^\infty(\Gamma)$ on the top ($+$) and bottom ($-$) of the screen (with $\Im{\lambda^\pm}\geq 0$ a.e.\ and $1/(\lambda^++\lambda^-)$ also bounded), we believe this to be false, since it relies on the incorrect assertion (see \cite[p215]{ben2013application}) that $L^\infty$ functions are multipliers on the spaces $H^{\pm 1/2}$ (again, see Remark \ref{rem:Comparison}). 
By using an integral equation formulation different to that in \cite{ben2013application}, and generalising that in \cite{kress2003integral}, we provide a corrected well-posedness analysis, which reveals that the assumption that $1/(\lambda^++\lambda^-)$ is bounded is unnecessary for BVP well-posedness. 
Finally, we point out that in \cite{ben2013application} and \cite{kress2003integral} the screen/crack is allowed to have non-zero curvature, whereas for simplicity we restrict our attention to the case where the screen/crack is planar (flat), as we did in \cite{ScreenPaper} and \cite{chandler-wilde2019} for the Dirichlet/Neumann cases. This simplifies the BIE formulation because the double layer and adjoint double layer operators ($K_\sigma$ and $K'_\sigma$ in \cite[Eqn~(9)]{ben2013application}) are identically zero on a flat screen. However, we expect that the generalisation to curved fractal screens (which we leave for future work) should not present too much difficulty (see Remark \ref{rem:Curved}). 

The structure of the paper is as follows. 
In \S\ref{sec:Prelim} we collect some key results concerning variational problems, Sobolev spaces and boundary integral operators that will be used throughout the paper. 
In \S\ref{sec:Problem} we state the impedance screen BVP and BIE and present our well-posedness analysis. 
In \S\ref{sec:BEM} we prove convergence of exact and numerical solutions on prefractal screens to solutions on limiting fractal screens using the framework of Mosco convergence. In \S\ref{sec:BEMegs} we consider some specific examples of fractal screens to which our analysis applies. In \S\ref{sec:Numerical} we give details of our BEM implementation and provide numerical results illustrating our theoretical predictions.

\section{Preliminaries\label{sec:Prelim}}
In this section we review some background theory and set our notation.
\subsection{Variational problems and Mosco convergence} \label{sec:Mosco}

Suppose that $H,\cH$ are Hilbert spaces, with $\cH$ a unitary realisation of $H^*$ with dual pairing $\langle\cdot,\cdot\rangle$, meaning that the map $u\mapsto\langle u,\cdot\rangle$ is a unitary isomorphism from $\cH$ to $H^*$. We recall 
that a bounded linear operator $A:H\to\cH$ is said to be \textit{coercive} if there exists $\alpha>0$ such that
\begin{equation*} 
|\langle Au,u\rangle|\ge \alpha\|u\|^2_H,
\quad \mbox{for all } u\in H.
\end{equation*}
We call $A$ {\em compactly perturbed coercive} if $A=A_0+A_1$ with $A_0$ coercive 
and $A_1$ compact. In this case, $A$ is invertible if and only if it is injective, by the standard Reisz-Fredholm theory. 

Let $W\subset H$ be a closed subspace of $H$ 
and let $\cW := (W^{a,\cH})^\perp\subset \cH$ (where $W^{a,\cH}$ is the annihilator of $W$ in $\cH$) be the unitary realisation of $W^*$ with dual pairing inherited from $\cH\times H$ (cf.\ \cite[Lem.~2.2]{ChaHewMoi:13}).  
We say $A$ is invertible on $W$ if $P_{\cW}A|_{W}$ is invertible, where $P_{\cW}$ is $\cH$-orthogonal projection onto $\cW$. Equivalently, $A$ is invertible on $W$ if, for every $f\in \cW$, there exists a unique $u_W\in W$ such that
\begin{equation} \label{eq:varW}
\langle Au_W,v\rangle = \langle f,v\rangle, \quad \mbox{for all } v\in W.
\end{equation}

When considering the approximability of the solution of the variational problem \eqref{eq:varW} by the solutions of variational problems posed on a sequence of closed subspaces $W_j\subset H$, 
an important tool is the notion of Mosco convergence.\footnote{
	Our definition of Mosco convergence coincides with that of \cite[Definition 1.1]{Mosco69}, specialised from the setting of convex subsets of a general Banach space to the setting of closed linear subspaces of a Hilbert space.
}
\begin{defn}
	\label{def:mosco} Let $W$ and $W_j$, for $j\in \N$, be closed subspaces of a Hilbert space $H$. We say that $W_j$ 
	Mosco-converges to $W$ (written $W_j\xrightarrow M W$) if
	\begin{itemize}
		\item[(i)] For  every $w\in W$ and $j\in \N$ there exists $w_j\in W_j$ such that $\|w_j- w\|_H\to 0$;
		\item[(ii)] If $(W_{j_m})$ is a subsequence of $(W_j)$, $w_m\in W_{j_m}$, for $m\in \N$, and  $w_m\rightharpoonup w$ as $m\to\infty$, then $w\in W$.
	\end{itemize}
\end{defn}
\begin{lem}[{\cite[Lemma 2.5]{chandler-wilde2019}}] \label{lem:dec3} Let $W$ and $W_j$, for $j\in \N$, be closed subspaces of a Hilbert space $H$ such that $W_j\xrightarrow M W$ 
	as $j\to\infty$. 
	Let $A:H\to\cH$ be compactly perturbed 
	coercive, 
	and invertible on $W$. 
	Then there exists $J\in \N$ such that $A$ is invertible on $W_j$ for $j\geq J$ and 
	\[
	\sup_{j\geq J} \left\|(P_{\cW_j}A|_{W_j})^{-1}\right\| <\infty.
	\]
	Further, if given $f\in \cH$ we define $u_W:=\left(P_{\cW}A|_{W}\right)^{-1}f$ and $u_{W_j}:=(P_{\cW_j}A|_{W_j})^{-1}f$ (for $j\geq J$)
	then $\|u_{W_j}-u_W\|_H\to 0$ as $j\to\infty$.
\end{lem}



\subsection{Sobolev spaces, traces and boundary integral operators} \label{sec:SobolevSpaces}

Our notation follows that of \cite{McLean} and \cite{ChaHewMoi:13}.
Let $m\in\N$.
For a subset $E\subset\R^m$ we denote its complement $E^c:=\R^m\setminus E$, its closure $\overline{E}$ and its interior $E^\circ$.
We denote the Hausdorff (fractal) dimension of $E$ by $\dimH{E}$. 
For $m>1$ we say that a non-empty open set $\Omega\subset \R^m$ is 
Lipschitz 
if its boundary $\partial\Omega$ can at each point be locally represented as the graph (suitably rotated) of a Lipschitz continuous function from $\R^{m-1}$ to $\R$, with $\Omega$ lying only on one side of $\partial\Omega$.
For a more detailed definition see, e.g., \cite[Defn~1.2.1.1]{Gri}.
For $m=1$, we say that $\Omega\subset \R$ is Lipschitz if $\Omega$ is a countable union of open intervals whose closures are disjoint and whose endpoints have no limit points.

For $s\in \R$, let $H^s(\R^m)$ denote the Hilbert space of tempered distributions
whose Fourier transforms (defined as $\hat{u}(\xi):= \frac{1}{(2\pi)^{m/2}}\int_{\R^m}\re^{-\ri \xi\cdot x}u(x)\,\rd x$ for $u\in C_0^\infty(\R^m)$) are locally integrable with
\[\|u\|_{H^s(\R^m)}^2:=\int_{\R^m}(1+|\xi|^2)^{s}\,|\hat{u}(\xi)|^2\,\rd \xi < \infty.\]
In particular,
$H^0(\R^m)=L^2(\R^m)$ with equal norms. The dual space of $H^s(\R^m)$ can be unitarily realised as $(H^s(\R^m))^*\cong H^{-s}(\R^m)$,
with dual pairing
\begin{align}\label{DualDef}
\left\langle u, v \right\rangle_{H^{-s}(\R^m)\times H^{s}(\R^m)}:= \int_{\R^m}\hat{u}(\xi) \overline{\hat{v}(\xi)}\,\rd \xi,
\end{align}
which coincides with the $L^2(\R^m)$ inner product when both $u$ and $v$ are in $L^2(\R^m)$.

For any non-empty open set $\Omega\subset\R^m$ and any closed set $F\subset \R^m$, we define
\begin{equation*} 
\tH^s(\Omega):=\overline{C^\infty_0(\Omega)}^{H^s(\R^m)}
\quad \text{and} \quad 
H_F^s := \{u\in H^s(\R^m): \supp{u} \subset F\}. 
\end{equation*}
Clearly $\tH^s(\Omega)\subset H^s_{\overline\Omega}$, and when $\Omega$ is sufficiently regular
it holds that $\tH^s(\Omega)=H^s_{\overline\Omega}$; 
however, in general these two closed subspaces of $H^s(\R^m)$ can differ \cite[\S3.5]{ChaHewMoi:13}. 
We also define 
\begin{align*}
H^s(\Omega)&:=\{U|_\Omega : U\in H^s(\R^m)\},\qquad
\|u\|_{H^{s}(\Omega)}:=\inf_{\substack{U\in H^s(\R^m)\\ U|_{\Omega}=u}}\normt{U}{H^{s}(\R^m)}.
\end{align*}
Although $H^s(\Omega)$ is a space of distributions on $\Omega$, it can be identified with a space of distributions on $\R^m$, namely $(H^s_{\Omega^c})^\perp\subset H^s(\R^m)$, where $^\perp$ denotes orthogonal complement in $H^s(\R^m)$; the restriction operator $|_\Omega :(H^s_{\Omega^c})^\perp\to H^s(\Omega)$ is a unitary isomorphism between the two spaces. 
A unitary realisation of $(\tH^s(\Omega))^*$, valid for arbitrary open $\Omega\subset\R^m$ (see \cite[Thm.~3.3]{ChaHewMoi:13}) is
\begin{align}
\label{isdual}
(\tH^s(\Omega))^*  \cong H^{-s}(\Omega)\cong (H^{-s}_{\Omega^c})^\perp  \quad \mbox{with} \quad
\langle u,v \rangle_{H^{-s}(\Omega)\times \tH^s(\Omega)}:=\langle U,v \rangle_{H^{-s}(\R^m)\times H^s(\R^m)},
\end{align}
where $U\in H^{-s}(\R^m)$ is any extension of $u\in H^{-s}(\Omega)$ with $U|_\Omega=u$. 
For bounded open $\Omega\subset \R^n$ and compact $K\subset \R^n$ the embeddings $H^t(\Omega)\subset H^s(\Omega)$ and $H^t_K\subset H^s_K$ are compact for $t>s$ \cite[Thm~3.27]{McLean}. 

Clearly $H^0(\Omega)=L^2(\Omega)$ with equal norms. The space $H^0(\Omega)$ is also unitarily isomorphic to the space $\widetilde{H}^0(\Omega)$,  
the unitary isomorphism being given by the 
zero extension operator $\,\widetilde{}:H^0(\Omega)\to \tH^0(\Omega)$ defined for $u\in H^0(\Omega)$ by $\widetilde{u}(x)=u(x)$ for $x\in \Omega$ and $\widetilde{u}(x)=0$ for $x\in \Omega^c$, whose inverse is the  
restriction operator $|_\Omega:\tH^0(\Omega)\to H^0(\Omega)$. 
However, we shall continue to use the notations $H^0(\Omega)$ and $\tH^0(\Omega)$, rather than simply writing $L^2(\Omega)$, in order to maintain the distinction between distributions on $\R^m$ and distributions on $\Omega$, since this distinction 
is crucial to many of our arguments.  
We recall that $L^\infty(\Omega)$ is a space of multipliers on $H^0(\Omega)$, with $\|\phi u\|_{H^0(\Omega)} \leq  \|\phi\|_{L^\infty(\Omega)}\|u\|_{H^0(\Omega)}$ for $\phi \in L^\infty(\Omega)$ and $u \in H^0(\Omega)$. 
We note also that the space $H^0_{\overline\Omega}$ can be similarly identified (by a unitary isomorphism) with $L^2(\overline\Omega)$, which itself can be identified with $L^2(\Omega)$ if and only if the Lebesgue measure of $\partial\Omega$ is zero.

For the screen scattering problem, we define Sobolev spaces on the hyperplane $\Gamma_\infty = \R^{n-1}\times\{0\}$ by associating $\Gamma_\infty$ with $\R^{n-1}$ (so in the definitions earlier in this section we are taking $m=n-1$) and setting $H^s(\Gamma_\infty):= H^s(\R^{n-1})$, for $s\in\R$.
For $E\subset\Gamma_\infty$ we set $\widetilde{E}:=\{\widetilde x\in\R^{n-1}: (\widetilde x,0)\in E\}\subset \R^{n-1}$.
Then for a closed subset $F\subset\Gamma_\infty$ we define $H^s_F:=H^s_{\widetilde{F}}\subset H^s(\Gamma_\infty)$, and for a (relatively) open subset $\Omega\subset \Gamma_\infty$ we set
$\tH^{s}(\Omega):=\tH^{s}(\widetilde{\Omega})\subset H^s(\Gamma_\infty)$
and $H^{s}(\Omega):=H^{s}(\widetilde{\Omega})$, etc.

In the exterior domain $D:=\R^n\setminus\overline{\Gamma}$ we work with Sobolev spaces defined via weak derivatives. 
Given a non-empty open $\Omega\subset \R^n$, let
$W^1(\Omega) := \{u\in L^2(\Omega): \nabla u \in L^2(\Omega)\}$
and let $W^{1,\mathrm{loc}}(\Omega)$ denote the ``local'' space in which square integrability of $u$ and $\nabla u$ is required only on bounded subsets of $\Omega$.
Note that $H^1(D)\subsetneqq W^1(D)$. 
We define
$U^+:=\{(x_1,\ldots,x_n)\in\R^n:x_n>0\}$ and $U^-:=\R^n\setminus\overline{U^+}$, and from the half spaces $U^\pm$ to the hyperplane $\Gamma_\infty$ we define the standard trace operators $\gamma^\pm:W^1(U^\pm)\to H^{1/2}(\Gamma_\infty)$ and $\partial_n^\pm:
\{u\in W^1(U^\pm): \Delta u \in L^2(U^\pm)\}
\to H^{-1/2}(\Gamma_\infty)$. 
We adopt the convention that the unit normal vector on $\Gamma_\infty$ points into $U^+$, so $\partial_n^\pm u^\pm = (\partial u^\pm/\partial x_n)|_{x_n=0}$ for smooth $u^\pm$ defined in $\overline{U^\pm}$. 
When applying $\gamma^\pm$ and $\partial_n^\pm $ to elements $u$ of the local space $W^{1,\mathrm{loc}}(D)$ (with $\Delta u$ locally in $L^2$ for $\partial_n^\pm $), we first pre-multiply $u$ by a cutoff 
$\chi\in C^\infty_0(\R^n)$ with $\chi=1$ in some neighbourhood of $\overline\Gamma$.  
But whenever an expression is independent of the choice of $\chi$ we omit the cutoff; this applies e.g.\ to the restrictions $\gamma^\pm u|_\Gamma:=\gamma^\pm\chi u|_\Gamma$ and $\partial_n^\pm u|_\Gamma:=\partial_n^\pm\chi u|_\Gamma$, and to the jumps $[u]:=\gamma^+\chi u-\gamma^-\chi u\in H^{1/2}_{\overline\Gamma}$ and $[\partial_n u]:=\partial_n^+\chi u-\partial_n^-\chi u\in H^{-1/2}_{\overline\Gamma}$, provided $u$ is continuously differentiable across $\Gamma_\infty \setminus \overline\Gamma$.

Let $\cS:H^{-1/2}_{\overline\Gamma}\to C^2(D)\cap W^{1,{\rm loc}}(\R^n)$ and $\cD:H^{1/2}_{\overline\Gamma}\to C^2(D)\cap W^{1,{\rm loc}}(D)$ be the single- and double-layer potentials, which for $\varphi\in C^\infty_0(\Gamma)$ and $x\in D$ satisfy
\begin{align}
\label{SLPDLP}
\cS\varphi(x)=\int_{\Gamma}\Phi(x,y)\varphi(y)\,\rd s(y), 
\qquad
\cD\varphi(x)=\int_{\Gamma}\frac{\partial\Phi(x,y)}{\partial n(y)}\varphi(y)\,\rd s(y),
\end{align}
with $\Phi(x,y)=\re^{\ri k |x-y|}/(4\pi |x-y|)$ ($n=3$) or $\Phi(x,y)=(\ri/4)H^{(1)}_0(k|x-y|)$ ($n=2$), 
and let $S$ and $T$ be the single-layer and hypersingular boundary integral operators 
defined by
\begin{align}
\label{BIOs}
S\varphi:=(\gamma^\pm\cS\varphi)|_{\Gamma},
\qquad
T\varphi&:=(\partial_n^\pm\cD\varphi)|_{\Gamma},
\end{align}
which extend to continuous operators $S:H^{s}_{\overline\Gamma}\to H^{s+1}(\Gamma)$ and $T:H^{s}_{\overline\Gamma}\to H^{s-1}(\Gamma)$ for all $s\in\R$.

\section{Problem statement and well-posedness analysis}
\label{sec:Problem}
Let $\Gamma$ be a bounded, relatively open subset of $\Gamma_\infty=\R^{n-1}\times\{0\}\subset\R^n$, $n=2,3$, and let $D:=\R^n\setminus\overline{\Gamma}$. 
The impedance screen BVP we study is stated below. Conditions \eqref{BVP3} and \eqref{BVP4} 
are non-standard assumptions that make the BVP well-posed for arbitrary bounded open screens $\Gamma\subset\Gamma_\infty$. 
As explained in \S\ref{sec:SobolevSpaces}, even though both spaces can be identified with $L^2(\Gamma)$, the motivation for our use of the notations $H^0(\Gamma)$ and $\tH^0(\Gamma)$ throughout this section is to maintain a clear distinction between functions/distributions on the screen $\Gamma$ and those on the hyperplane $\Gamma_\infty$.  

\begin{defn}[Impedance screen BVP]
\label{def:BVP}
Let $k>0$ and 
$\lambda^\pm\in L^\infty(\Gamma)$ with $\Im[\lambda^\pm]\geq 0$ a.e.\ on $\Gamma$.
Given $g^\pm\in H^{0}(\Gamma)\left(= L^2(\Gamma)\right)$, 
find $u\in C^2(D)\cap W^{1,{\rm loc}}(D)$ such that 
\begin{align}
\label{eq:HE}
\Delta u + k^2 u = 0, &\qquad \text{in }D,\\
\label{eq:SRC}
\partial_r u - iku = o(r^{-(n-1)/2}), &\qquad \text{uniformly as }r=|x|\to\infty,\\\label{eq:BC}
\partial^\pm_n u|_\Gamma \pm \lambda^\pm \gamma^\pm u|_\Gamma = g^\pm, &\qquad \text{on }\Gamma,\\
\label{BVP3}
[u]\in \widetilde{H}^{1/2}(\Gamma),\\
\label{BVP4}
[\partial_n u]\in \widetilde{H}^0(\Gamma)\left(\cong L^2(\Gamma)\right).
\end{align}
To model scattering of an incident wave $u^i$ (a solution of the Helmholtz equation in a neighbourhood of $\Gamma$) 
we would take $g^\pm=-(\partial^\pm_n u^i|_\Gamma \pm \lambda^\pm \gamma^\pm u^i|_\Gamma)$ 
and interpret $u$ as the scattered field.
\end{defn}

Our main result in this section is the following.
\begin{thm}
	\label{Thm1}
	The BVP \eqref{eq:HE}-\eqref{BVP4} in Definition \ref{def:BVP} is well-posed for any bounded open screen $\Gamma\subset\Gamma_\infty$. 
\end{thm}

To prove Theorem \ref{Thm1} we follow the same classical argument as in \cite{ben2013application}: we first prove BVP uniqueness (Theorem \ref{Thm2}), then BVP-BIE equivalence (Theorem \ref{BVPBIEequiv}), then BIE existence via Riesz-Fredholm theory (proving that the boundary integral operator is Fredholm of index zero and that it is injective, as a consequence of BVP uniqueness and BVP-BIE equivalence). However, the details of each stage differ from \cite{ben2013application} because the arguments in \cite{ben2013application} apply only to sufficiently smooth $\Gamma$ and $\lambda^\pm$ (see Remark \ref{rem:Comparison}). 

Before embarking on the proof of Theorem \ref{Thm1} we pause to reassure the reader that when $\Gamma$ is sufficiently smooth the additional conditions \eqref{BVP3} and \eqref{BVP4} are redundant, in the sense that the BVP is well-posed without them. 

\begin{thm}
	\label{Thm2}
	If $\widetilde{H}^{1/2}(\Gamma)=H^{1/2}_{\overline\Gamma}$ then condition \eqref{BVP3} is redundant. 
	If $H^{-1/2}_{\partial \Gamma}=\{0\}$ then condition \eqref{BVP4} is redundant. In particular, if $\Gamma$ is Lipschitz except at a countable set of points with finitely many limit points, then both conditions are redundant and the BVP \eqref{eq:HE}-\eqref{eq:BC} is well-posed.
\end{thm}
\begin{proof}
	By the assumption that $u\in C^2(D)$ we know that $[u]  \in H^{1/2}_{\overline\Gamma}$ and $[\partial_n u]  \in H^{-1/2}_{\overline\Gamma}$. Hence the first statement is obvious. For the second statement, from \eqref{eq:BC} we have that $[\partial_n u]|_\Gamma=\partial^+_n u|_\Gamma-\partial^-_n u|_\Gamma= -(\lambda^+ \gamma^+ u|_\Gamma + \lambda^- \gamma^- u|_\Gamma)+(g^+-g^-)\in H^{0}(\Gamma)$. This means the zero extension $\widetilde{[\partial_n u]|_\Gamma}$ is well-defined as an element of $\widetilde{H}^0(\Gamma)\subset \widetilde{H}^{-1/2}(\Gamma)\subset H^{-1/2}_{\overline\Gamma}$. But the assumption that $H^{-1/2}_{\partial \Gamma}=\{0\}$ is equivalent to the restriction operator $|_\Gamma:H^{-1/2}_{\overline\Gamma}\to H^{-1/2}(\Gamma)$ being injective, from which it follows that $[\partial_n u]=\widetilde{[\partial_n u]|_\Gamma}$, so that $[\partial_n u]\in \widetilde H^0(\Gamma)$. The final statement then follows from the first two statements combined with \cite[Theorem 3.24 and Lemma 3.10(xi)]{ChaHewMoi:13}. 
\end{proof}
\vspace{3mm}

We now prove BVP uniqueness for \eqref{eq:HE}-\eqref{BVP4}. In \cite{ben2013application} uniqueness is quoted without proof from \cite{kress2003integral}, where a proof is given in the case where $\lambda^+=\lambda^-$, $n=2$, and $\Gamma$ is a single smooth arc (a line segment in the flat case). The proof in \cite{kress2003integral} involves carefully cutting off the singular endpoint contributions. But we can prove uniqueness for general open $\Gamma$ in the general case where $\Im[\lambda^\pm]\geq 0$ using distributional calculus. The assumptions \eqref{BVP3} and \eqref{BVP4} are invoked to show that certain dual pairings can be replaced by $L^2(\Gamma)$ inner products. 

\begin{thm}
	\label{thm:BVPuniqueness}
	There is at most one solution to the BVP \eqref{eq:HE}-\eqref{BVP4}.
\end{thm}
\begin{proof}
	Suppose that $u$ is a solution of the homogeneous BVP with $g^\pm=0$. From \eqref{eq:BC} we have that $\partial_n^\pm u|_\Gamma = \mp\lambda^\pm \gamma^\pm u|_\Gamma \in H^{0}(\Gamma)$. This means the zero extensions $\widetilde{\partial_n^\pm u|_\Gamma}$ are well-defined as elements of $\widetilde{H}^0(\Gamma)\subset \widetilde{H}^{-1/2}(\Gamma)\subset H^{-1/2}_{\overline\Gamma}$. 
	Let $B_R\subset \R^n $ denote an open ball of radius $R>0$ centered at the origin such that $\overline\Gamma\subset B_R$, and let $B_R^\pm=\{x\in B_R:\pm x_n>0\}$ denote the corresponding upper and lower open half-balls. 
	Let $\chi\in C^\infty_0(B_R)$, with $\chi=1$ in a neighbourhood of $\overline\Gamma$. Apply Green's first identity 
	in $B_R^\pm$ to the decomposition $u = \chi u + (1-\chi)u$ and sum the results to get that 
	\[ \int_{B_R^+\cup B_R^-} |\nabla u|^2 - k^2 |u|^2 = \int_{\partial B_R} \frac{\partial u}{\partial \nu} \overline{u} - \left \langle \partial_n^+\chi u,\gamma^+\chi u\right \rangle_{H^{-1/2}(\Gamma_\infty)\times H^{1/2}(\Gamma_\infty)} + \left \langle \partial_n^-\chi u,\gamma^-\chi u\right \rangle_{H^{-1/2}(\Gamma_\infty)\times H^{1/2}(\Gamma_\infty)},    \]
	where $\nu$ is the unit outward normal on $\partial B_R$.
	Since the LHS is real, taking conjugates and then imaginary parts gives
	\begin{align*}
	\label{}
	\Im\int_{\partial B_R} \overline{\frac{\partial u}{\partial \nu}}u 
	&= \Im\left[ - \left \langle \partial_n^+\chi u,\gamma^+\chi u\right \rangle_{H^{-1/2}(\Gamma_\infty)\times H^{1/2}(\Gamma_\infty)} + \left \langle \partial_n^-\chi u,\gamma^-\chi u\right \rangle_{H^{-1/2}(\Gamma_\infty)\times H^{1/2}(\Gamma_\infty)} \right].
	\end{align*}
	Adding and subtracting $\left \langle \partial_n^+\chi u,\gamma^-\chi u\right \rangle_{H^{-1/2}(\Gamma_\infty)\times H^{1/2}(\Gamma_\infty)}$ on the RHS, and invoking condition \eqref{BVP3} to justify the replacement of $\partial_n^+\chi u$ by $\widetilde{\partial_n^+ u|_\Gamma}$, gives
	\begin{align*}
	\label{}
	\Im\int_{\partial B_R} \overline{\frac{\partial u}{\partial \nu}}u 
	&= \Im\left[ - \left \langle \partial_n^+\chi u,[u]\right \rangle_{H^{-1/2}(\Gamma_\infty)\times H^{1/2}(\Gamma_\infty)} - \left \langle [\partial_n u],\gamma^-\chi u\right \rangle_{H^{-1/2}(\Gamma_\infty)\times H^{1/2}(\Gamma_\infty)} \right]\\
	&= \Im\left[ - \left \langle \widetilde{\partial_n^+ u|_\Gamma},[u]\right \rangle_{H^{-1/2}(\Gamma_\infty)\times H^{1/2}(\Gamma_\infty)} - \left \langle [\partial_n u],\gamma^-\chi u\right \rangle_{H^{-1/2}(\Gamma_\infty)\times H^{1/2}(\Gamma_\infty)} \right].
	\end{align*}
	All the arguments in the dual pairings are now $L^2$ functions, by the discussion about $\widetilde{\partial_n^+ u|_\Gamma}$ above, and by condition \eqref{BVP4}, which implies in particular that $[\partial_n u]=\widetilde{[\partial_n u]|_\Gamma}$. Hence, by applying the boundary condition \eqref{eq:BC}, we find that
	\begin{align*}
	\label{}
	\Im\int_{\partial B_R} \overline{\frac{\partial u}{\partial \nu}}u 
	&= \Im\left[ - \left(\widetilde{\partial_n^+ u|_\Gamma},[u]\right)_{L^2(\Gamma_\infty)} - \left(\widetilde{[\partial_n u]|_\Gamma},\gamma^-\chi u\right)_{L^2(\Gamma_\infty)} \right]\\
	&= \Im\left[ - \left(\partial_n^+ u|_\Gamma,[u]|_\Gamma\right)_{L^2(\Gamma)} - \left([\partial_n u]|_\Gamma,\gamma^- u|_\Gamma\right)_{L^2(\Gamma)} \right]\\
	&= \Im\left[ \left(\lambda^+\gamma^+ u|_\Gamma,\gamma^+u|_\Gamma\right)_{L^2(\Gamma)} + \left(\lambda^-\gamma^- u|_\Gamma,\gamma^- u|_\Gamma\right)_{L^2(\Gamma)} \right]\\
	&= \int_\Gamma \Im\left[ \lambda^+\right]\left|\gamma^+ u|_\Gamma\right|^2 + \int_\Gamma \Im\left[ \lambda^-\right]\left|\gamma^- u|_\Gamma\right|^2.
	\end{align*}
	Therefore $\Im\int_{\partial B_R} \overline{\frac{\partial u}{\partial \nu}}u \geq 0 $, by our assumption that $\Im\left[ \lambda^\pm\right]\geq 0$ a.e. on $\Gamma$, so by Rellich's Lemma \cite[Thm~2.13]{CoKr:98} we get $u=0$ in $\R^n\setminus \overline{B_R}$, and then by unique continuation \cite[Thm~8.6]{CoKr:98} we get $u=0$ in the whole of $D=\R^n\setminus\overline\Gamma$. 
\end{proof}


Next we establish BVP-BIE equivalence.

\begin{thm}
	\label{BVPBIEequiv}
	If $u$ solves the BVP \eqref{eq:HE}-\eqref{BVP4} then
	\begin{align}
	\label{eq:Rep}
	u = \cD \phi - \cS \psi,
	\end{align}
	where $\phi=[u]\in \widetilde{H}^{1/2}(\Gamma)$ and $\psi=[\partial_n u]\in \widetilde{H}^{0}(\Gamma)$ satisfy the 
	integral equation 
	\begin{align}
	\label{BIE}
	A \left(\begin{array}{c}{ \phi}\\ {\psi}\end{array}\right) = \left(\begin{array}{c} 
	-(g^++g^-)
	\\[3mm] g^+-g^-\end{array}\right), 
	\end{align}
	with $A:\widetilde{H}^{1/2}(\Gamma)\times \widetilde{H}^{0}(\Gamma) \to H^{-1/2}(\Gamma)\times H^0(\Gamma)$ given by the matrix of operators
	\begin{equation}\label{eq:A_def} 
	A  = \left(\begin{array}{cc} -\frac{1}{2}(\lambda^++\lambda^-)|_\Gamma - 2T & (\lambda^+-\lambda^-)S \\[4mm] \frac{1}{2}(\lambda^+-\lambda^-)|_\Gamma & |_\Gamma - (\lambda^+ +\lambda^-)S \end{array}\right).
	\end{equation}
(Here the notation $(\lambda^+\pm\lambda^-)|_\Gamma$ denotes restriction to $\Gamma$ followed by multiplication by $(\lambda^+\pm\lambda^-)$, i.e.\ $(\lambda^+\pm\lambda^-)|_\Gamma (\phi):=(\lambda^+\pm\lambda^-)\phi|_\Gamma$ for $\phi\in \tH^{1/2}(\Gamma)$.)

	Conversely, suppose that $\phi\in\widetilde{H}^{1/2}(\Gamma)$ and $\psi\in\widetilde{H}^{0}(\Gamma)$ satisfy \eqref{BIE}. Then $u$ defined by \eqref{eq:Rep} satisfies the BVP \eqref{eq:HE}-\eqref{BVP4}, and $[u]=\phi$ and $[\partial_n u]=\psi$.
\end{thm}

\begin{proof}
	We note first that if $u$ is any element of $C^2(D)\cap W^{1,{\rm loc}}(D)$ satisfying \eqref{eq:Rep} for some 
	$\phi\in H^{1/2}_{\overline \Gamma}$ and $\psi\in H^{-1/2}_{\overline \Gamma}$ then multiplying \eqref{eq:Rep} by a cut-off function $\chi\in C^\infty_0(\R^n)$, with $\chi=1$ in a neighbourhood of $\overline\Gamma$, 
	and taking Dirichlet and Neumann traces onto $\Gamma_\infty$, gives
	\begin{align}
	\label{Rep0}
	\gamma^\pm \chi u &= \pm \frac12 \phi - \gamma^\pm \chi \cS \psi,\\
	\label{Rep1}
	\partial_n^\pm \chi u &= \partial_n^\pm \chi\cD\phi \pm \frac12 \psi,
	\end{align}
after application of the relations $\gamma^\pm \chi\cD\phi = \pm\frac12 \phi $ and $\partial_n^\pm \chi\cS\psi = \mp\frac12 \psi $ \cite[(28)-(29)]{ScreenPaper}. 
Then, restricting to $\Gamma$ and applying \eqref{BIOs}, we obtain
	\begin{alignat}{3}
	\label{eq:DirEq}
	\gamma^\pm u|_\Gamma &= \pm \frac12 \phi|_\Gamma- S \psi, && \qquad(\text{equality in }H^{1/2}(\Gamma))\\
	\label{eq:NeuEq}
	\partial_n^\pm u|_\Gamma &= T\phi \pm \frac12 \psi|_\Gamma. && \qquad(\text{equality in }H^{-1/2}(\Gamma))
	\end{alignat}
We shall use these relations in both directions of the proof.	
	
	Now suppose that $u$ satisfies the BVP \eqref{eq:HE}-\eqref{BVP4}. Then \eqref{eq:Rep} with $\phi=[u]$ and $\psi=[\partial_n u]$ follows from \eqref{eq:HE} and \eqref{eq:SRC} and Green's representation theorem for screens (see e.g.~\cite[Thm 3.2]{CoercScreen}), and the facts that $\phi\in \widetilde{H}^{1/2}(\Gamma)$ and $\psi\in \widetilde{H}^{0}(\Gamma)$ are precisely the conditions \eqref{BVP3} and \eqref{BVP4}. 
	The boundary condition \eqref{eq:BC} tells us that $\partial_n^\pm u|_\Gamma\in H^0(\Gamma)$, and hence, by \eqref{eq:NeuEq}, that $T\phi\in H^0(\Gamma)$, so that \eqref{eq:NeuEq} is actually an equality in $H^0(\Gamma)$. Then, combining \eqref{eq:DirEq} and \eqref{eq:NeuEq} with \eqref{eq:BC} we obtain
	\begin{alignat}{3}
	\label{Rep3a}
T\phi \pm \frac{1}{2}\psi|_\Gamma + \frac{1}{2}\lambda^\pm \phi|_\Gamma \mp \lambda^\pm S\psi=g^\pm,&& \qquad(\text{equality in }H^{0}(\Gamma))
	\end{alignat}
and taking first the sum and then the difference of the $+$ and $-$ versions of \eqref{Rep3a} produces \eqref{BIE}.
%
	
For the converse, suppose that $\phi\in\widetilde{H}^{1/2}(\Gamma)$ and $\psi\in\widetilde{H}^{0}(\Gamma)$ satisfy \eqref{BIE}. Then $u$ defined by \eqref{eq:Rep} lies in $C^2(D)\cap W^{1,{\rm loc}}(D)$ and satisfies \eqref{eq:HE} and \eqref{eq:SRC}. It also satisfies \eqref{Rep0}-\eqref{Rep1}, and then since $[\cS \psi]=[\partial_n \cD \phi]=0$ (by \cite[(27) and (30)]{ScreenPaper}) it follows that $[u]=\phi$ and $[\partial_n u]=\psi$, and hence that \eqref{BVP3} and \eqref{BVP4} hold. The boundary condition \eqref{eq:BC} follows by combining \eqref{BIE} with \eqref{eq:DirEq} and \eqref{eq:NeuEq}.

\end{proof}

To prove BIE existence, we first establish that the operator $A$ is compactly perturbed coercive. For this we note that the range space $H^{-1/2}(\Gamma)\times H^{0}(\Gamma)$ of $A$ provides a unitary realisation of the dual space of the domain space $\widetilde{H}^{1/2}(\Gamma)\times \widetilde{H}^{0}(\Gamma)$ of $A$ via the dual pairing
\begin{align}
\label{eq:Dual}
\left\langle \left(\begin{array}{c}{u}\\ {u'}\end{array}\right), \left(\begin{array}{c}{ v}\\ v'\end{array}\right)\right\rangle = \left\langle u,v\right\rangle_{H^{-1/2}(\Gamma)\times \widetilde{H}^{1/2}(\Gamma)} + \left\langle u',v'\right\rangle_{H^{0}(\Gamma)\times \widetilde{H}^{0}(\Gamma)}, 
\end{align}
where the second pairing on the RHS coincides with the $L^2$ inner product $(u',v'|_\Gamma)_{L^2(\Gamma)}$.

\begin{thm}
\label{thm:CoercCompact}
The operator $A:\widetilde{H}^{1/2}(\Gamma)\times \widetilde{H}^{0}(\Gamma) \to H^{-1/2}(\Gamma)\times H^0(\Gamma)$ defined in \eqref{eq:A_def} is compactly perturbed coercive (and hence Fredholm of index zero) with respect to the dual pairing $\langle\cdot,\cdot\rangle$.
\end{thm}

\begin{proof}
	We decompose $A$ as
	\[A = A_0 + A_1, \] 
	where 
	\[ A_0 = \left(\begin{array}{cc}-2T_0 & 0 \\ 0 & |_\Gamma  \end{array}\right),
	\qquad
	A_1  = \left(\begin{array}{cc}  -\frac{1}{2}(\lambda^++\lambda^-)|_\Gamma - 2(T-T_0) & (\lambda^+-\lambda^-)S \\[3mm] \frac{1}{2}(\lambda^+-\lambda^-)|_\Gamma &  - (\lambda^+ +\lambda^-)S \end{array}\right),\]
with $T_0:\widetilde{H}^{1/2}(\Gamma)\to H^{-1/2}(\Gamma)$ denoting the hypersingular boundary integral operator for the Laplace equation ($k=0$). 
	
	To prove that $A_0$ is coercive, we recall 
that $T_0$ is coercive: 
	there exists $c>0$ such that
	\[\Re\left[\left\langle -T_0 \phi,\phi\right\rangle_{{H^{-1/2}(\Gamma)\times \widetilde{H}^{1/2}(\Gamma)}}\right]\geq c \|\phi\|_{\widetilde H^{1/2}(\Gamma)}^2, \qquad \phi\in \widetilde{H}^{1/2}(\Gamma).\]
	This result is well-known even for curved screens (see e.g.\ \cite[p.216]{ben2013application}), and for planar screens can be proved using the Fourier transform representation 
	\[\left\langle -T_0 \phi,\phi\right\rangle_{{H^{-1/2}(\Gamma)\times \widetilde{H}^{1/2}(\Gamma)}} = \frac{1}{2}\int_{\R^{n-1}}|\xi|\,|\widehat{\phi}(\xi)|^2\,\rd\xi, \qquad \phi\in \widetilde{H}^{1/2}(\Gamma),\]
which can be obtained as the $k\to 0$ limit of the corresponding formula for the $k>0$ case presented in \cite[(37)]{CoercScreen2}. 
(This also reveals that $\left\langle -T_0 \phi,\phi\right\rangle_{{H^{-1/2}(\Gamma)\times \widetilde{H}^{1/2}(\Gamma)}}>0$ in the planar case.)
	Together, these observations imply that
	\begin{align*}
	\label{}
\Re\left[\left\langle A_0 \left(\begin{array}{c}{\phi}\\ {\psi}\end{array}\right),\left(\begin{array}{c}{\phi}\\ {\psi}\end{array}\right)\right\rangle \right]
		&= \Re\left[\left\langle -2T_0 \phi,\phi\right\rangle_{H^{-1/2}(\Gamma)\times \widetilde{H}^{1/2}(\Gamma)}\right] + \left\langle \psi|_\Gamma,\psi\right\rangle_{H^{0}(\Gamma)\times \widetilde{H}^{0}(\Gamma)}
	\\
	&\geq \min\left\{2c,1\right\}\left(\|\phi\|_{\widetilde{H}^{1/2}(\Gamma)}^2 + \|\psi\|_{\widetilde H^{0}(\Gamma)}^2\right),
	\end{align*}
	so $A_0$ is coercive, as claimed.
	
	To show $A_1$ is compact it suffices to show that each operator appearing in $A_1$ is compact between the appropriate pair of spaces. In each case this follows by expressing the operator in question as a composition of bounded operators and a compact embedding. The requisite facts are that 
$|_\Gamma:\widetilde H^{s}(\Gamma)\to H^{s}(\Gamma)$,   $S:\widetilde H^{s}(\Gamma)\to H^{s+1}(\Gamma)$ and $T-T_0:\widetilde H^{s}(\Gamma)\to H^{s+1}(\Gamma)$ are all bounded for $s\in\R$, multiplication by an $L^\infty(\Gamma)$ function is bounded on $H^0(\Gamma)$, and that $H^s(\Gamma)\subset H^{t}(\Gamma)$ is a compact embedding for $s>t$ (see \S\ref{sec:SobolevSpaces}). The claimed boundedness of $T-T_0$ and $S$ between the stated spaces for arbitrary bounded open $\Gamma$ follows from the corresponding well-known statements for the smooth case, 
since if $\Gamma_\dag$ is a bounded open set with smooth boundary (e.g.~a ball) containing $\Gamma$ and $R:\tilde{H}^s(\Gamma_\dag)\to H^t(\Gamma_\dag)$ is bounded for some $s,t\in\R$ then $|_\Gamma\circ R:\tilde{H}^s(\Gamma)\subset\tilde{H}^s(\Gamma_\dag)\to H^t(\Gamma)$ is also bounded. For the flat screens we consider, these boundedness results can also be proved directly using Fourier representations of the operators $T-T_0$ and $S$. The case of $S$ is covered by \cite[Thm~1.6]{CoercScreen2}, and for the case of $T-T_0$ we refer to \cite[Proof of Thm~1.8]{CoercScreen2}, noting that, where $Z(\xi)=\sqrt{k^2-|\xi|^2}$ for $|\xi|\leq k$ and $Z(\xi)=\ri\sqrt{|\xi|^2-k^2}$ for $|\xi|> k$ (cf. \cite[Eqn~(33)]{CoercScreen2}), the underlying Fourier integral operator associated with $T-T_0$ has symbol $\frac{\ri}{2}(Z(\xi)-\ri|\xi|)$, which is $O(1/|\xi|)$ as $|\xi|\to \infty$. 
		%
	%
\end{proof}

We can now deduce the invertibility of $A$, and hence BIE existence.

\begin{thm}
\label{thm:BIEwellposed}
The operator $A:\widetilde{H}^{1/2}(\Gamma)\times \widetilde{H}^{0}(\Gamma) \to H^{-1/2}(\Gamma)\times H^0(\Gamma)$ defined in \eqref{eq:A_def} is invertible. 
\end{thm}
\begin{proof}
The integral operator is injective (as a consequence of Theorems \ref{BVPBIEequiv} and \ref{Thm2}) and Fredholm of index zero (by Theorem \ref{thm:CoercCompact}), hence also surjective.
\end{proof}

The proof of Theorem \ref{Thm1} then follows in the standard way.

\begin{proof}[Proof of Theorem \ref{Thm1}] 
Uniqueness of the BVP solution was proved in Theorem \ref{Thm2}. Existence of the BVP solution follows from Theorems \ref{BVPBIEequiv} and \ref{thm:BIEwellposed}.
\end{proof}

In the following remark we compare our results with those of \cite{ben2013application}. 

\begin{rem}
\label{rem:Comparison}
In \cite{ben2013application} the authors study the BVP \eqref{eq:HE}-\eqref{eq:BC} and the associated BIE 
	\begin{align}
	\label{BIE_Haddar}
	\widetilde{A} \left(\begin{array}{c}{ \phi}\\ {\psi}\end{array}\right) = \left(\begin{array}{c} \lambda^-g^++\lambda^+g^-\\[3mm] g^+-g^-\end{array}\right), 
	\end{align}
where 
	\begin{equation}\label{eq:A_def_Haddar} 
	\widetilde{A}  = \left(\begin{array}{cc} \lambda^+\lambda^-|_\Gamma + (\lambda^++\lambda^-)T & -\frac{1}{2}(\lambda^+-\lambda^-)|_\Gamma \\[4mm] \frac{1}{2}(\lambda^+-\lambda^-)|_\Gamma & |_\Gamma - (\lambda^+ +\lambda^-)S \end{array}\right),
	\end{equation}
which can be obtained formally from our BIE \eqref{BIE} by replacing the first row in \eqref{BIE} by $-(\lambda^++\lambda^-)/2$ times the first row plus $-(\lambda^+-\lambda^-)/2$ times the second row. 
One attraction of \eqref{BIE_Haddar} compared to \eqref{BIE} is that it contains one fewer instance of the operator $S$, which might make numerical implementation slightly cheaper. However, the presence of the factor $(\lambda^++\lambda^-)$ premultiplying $T$ in the $(1,1)$ entry of $\widetilde{A}$ introduces a spurious ill-posedness in the case $\lambda^++\lambda^-=0$, leading the authors of \cite{ben2013application} to require an additional assumption on $\lambda^\pm$ in their BVP well-posedness analysis \cite[\S2]{ben2013application}, namely that $(\lambda^++\lambda^-)^{-1}\in L^\infty(\Gamma)$. Our results (cf.\ Theorem \ref{Thm1}) show that this assumption is unnecessary for BVP well-posedness, and is associated purely with the choice of BIE made in \cite{ben2013application}.

The analysis in \cite{ben2013application} has further deficiencies, which we remark on here for completeness. In \cite{ben2013application} it is claimed that $\widetilde{A}$ is bounded as a mapping 
$\widetilde{A}:\widetilde{H}^{1/2}(\Gamma)\times \widetilde{H}^{-1/2}(\Gamma) \to H^{-1/2}(\Gamma)\times \widetilde{H}^{-1/2}(\Gamma)|_\Gamma$, presumably with $\widetilde{H}^{-1/2}(\Gamma)|_\Gamma$ normed by the norm inherited from $\widetilde{H}^{-1/2}(\Gamma)$ (although this is not explicitly stated in \cite{ben2013application}), and that the BVP \eqref{eq:HE}-\eqref{eq:BC} and BIE \eqref{BIE_Haddar} are well posed under the assumption that $\lambda^\pm\in L^\infty(\Gamma)$ with $\Im[\lambda^\pm]\geq 0$ and $(\lambda^++\lambda^-)^{-1}\in L^\infty(\Gamma)$, and $g^\pm\in H^{-1/2}(\Gamma)$ with $g^+-g^-\in \tH^{-1/2}(\Gamma)|_\Gamma$. 
However, we believe the justification of these claims in \cite{ben2013application} to be incomplete, even for smooth (e.g.\ Lipschitz) $\Gamma$. 
This is because the analysis in \cite{ben2013application} appears to rely on the incorrect assertion (see \cite[p215]{ben2013application}) that $L^\infty(\Gamma)$ functions are multipliers on the space $H^{-1/2}(\Gamma)$, which is false in general.\footnote{For an illustrative example in the context of $\R^m$ let $u\in C^\infty_0(\R^m)$ and let $\phi$ be the characteristic function of $\{x\in\R^m:x_m>0\}$. Then $\phi\in L^\infty(\R^m)$ and $u\in H^{1/2}(\R^m)$ but $\phi u\not\in H^{1/2}(\R^m)$ since $H^{1/2}(\R^m)$ functions cannot have jump discontinuities across $(m-1)$-dimensional manifolds. Thus $\phi$ is not a multiplier on $H^{1/2}(\R^m)$, and hence, by duality, not on $H^{-1/2}(\R^m)$ either.}
This means that, in particular, the operator $\widetilde{A}$ does not have the claimed mapping properties, and the first entry in the right-hand side of \eqref{BIE_Haddar} is not in $H^{-1/2}(\Gamma)$ (as claimed in \cite{ben2013application}) for general $g^\pm\in H^{-1/2}(\Gamma)$.

Provided that $\lambda^\pm$ and $(\lambda^++\lambda^-)^{-1}$ are sufficiently regular to be multipliers on $H^{\pm 1/2}(\Gamma)$, which holds e.g.\ when they are elements of $W^{1,\infty}(\Gamma)$ (cf.\ \cite[Theorem 3.20]{McLean}), the well-posedness analysis of \cite{ben2013application} appears to be valid, albeit under an implicit regularity assumption on $\Gamma$. No explicit regularity assumption on $\Gamma$ is stated in \cite{ben2013application}, but 
to see where the analysis breaks down for non-smooth $\Gamma$, we note that in \cite{ben2013application} the fact that the operator $\widetilde{A}$ is Fredholm of index zero is established by decomposing \[\widetilde{A} = \widetilde A_0 + \widetilde A_1, \] 
	where 
	\[ \widetilde A_0 = \left(\begin{array}{cc} (\lambda^+ +\lambda^-)T & -\frac{1}{2}(\lambda^+-\lambda^-)|_\Gamma \\ 0 & |_\Gamma  \end{array}\right),
	\qquad
	\widetilde A_1 = \left(\begin{array}{cc} \lambda^+\lambda^-|_\Gamma & 0 \\ \frac{1}{2}(\lambda^+-\lambda^-)|_\Gamma &  - (\lambda^+ +\lambda^-)S \end{array}\right),\]
	and showing that $\widetilde A_0$ is invertible and $\widetilde A_1$ is compact. 
A necessary condition for $\widetilde A_0$ to be invertible is that $|_\Gamma:\widetilde{H}^{-1/2}(\Gamma) \to H^{-1/2}(\Gamma)$ is injective. This is satisfied for smooth (e.g. Lipschitz) $\Gamma$, but fails in general. In particular, it fails when $\widetilde{H}^{-1/2}(\Gamma)=H^{-1/2}_{\overline\Gamma}$ but $H^{-1/2}_{\partial \Gamma}\neq\{0\}$, as is the case for the examples studied in \S\ref{sec:BEMegs}, including the classical Koch snowflake. 

In our analysis we work with an augmented BVP \eqref{eq:HE}-\eqref{BVP4} and a modified BIE \eqref{BIE} in a different function space setting, which requires us to assume extra regularity of the data $g^\pm$, namely that $g^\pm\in H^0(\Gamma)$. But this allows us to prove BVP and BIE well-posedness for general $\lambda^\pm\in L^\infty(\Gamma)$ (with $\Im[\lambda^\pm]\geq 0$ a.e.) and for arbitrary bounded open $\Gamma\subset\Gamma_\infty$. 
Furthermore, our extra assumption that $g^\pm\in H^0(\Gamma)$ is completely natural when it comes to scattering problems, since if the incident wave $u^i$ is a solution of the Helmholtz equation in a neighbourhood of $\Gamma$ then by elliptic regularity it is $C^\infty$ in a neighbourhood of $\Gamma$, which means that $g^\pm\in H^{0}(\Gamma)$ automatically. 


\end{rem}


\begin{rem}
\label{rem:Curved}
While our analysis is presented for flat (planar) screens, we expect that the extension to the curved case, where $\Gamma$ is a relatively open subset of the boundary of a Lipschitz (or possibly smoother) open set, should be possible in principle. Indeed, the only change to the boundary integral operator $A$ would be the addition of two extra off-diagonal terms $-2K'$ and $(\lambda^++\lambda^-)K$ in the $(1,2)$ and $(2,1)$ entries respectively (cf.\ the corresponding formula \cite[Equation (9)]{ben2013application} for the case of the operator $\widetilde{A}$ discussed in Remark \ref{rem:Comparison}), where $K$ and $K'$ are the double-layer and adjoint double-layer boundary integral operators, which are compact between the relevant spaces and so do not affect the analysis. But we leave this extension to future work.
\end{rem}



\section{Boundary element method}
\label{sec:BEM}
In this section we present and analyse a Galerkin BEM approach for the approximate solution of the integral equation \eqref{BIE} in the case where $\Gamma$ is a arbitrary bounded open subset of $\Gamma_\infty=\R^{n-1}\times\{0\}$. The case when $\partial\Gamma$ is regular (e.g.\ $\Gamma$ is a simple polygon) is classical \cite{ben2013application}. Here our focus is on the case when $\partial\Gamma$ is non-regular, even fractal. 

Our approach is to approximate the screen $\Gamma$, in a sense made clear in \eqref{MoscoPrefract} below, by a sequence of regular bounded open screens $\Gamma_j\subset\Gamma_\infty$, $j\in\N_0$, on which BEM approximation is possible with a suitable triangulation. When the limiting screen $\Gamma$ has a fractal boundary we call $\Gamma_j$ a \emph{prefractal} approximation. Our aim is to prove convergence of the BEM solution on $\Gamma_j$ to the true solution of the BIE \eqref{BIE} on $\Gamma$ in the joint limit as $j\to \infty$ with the BEM mesh width tending to zero. For convenience we assume that $\Gamma$ and $\Gamma_j$, $j\in\N_0$, are all contained in some larger bounded open screen $\Gamma^\dag\subset \Gamma_\infty$.

Theorem \ref{th:NA}
provides a general theoretical framework from which one can deduce such convergence results. We state the result in abstract terms since we expect it may be of wider interest beyond the current study. But the notation we use is chosen with Theorems \ref{thm:impedanceconv} and \ref{thm:impedanceNA} in mind, where we apply Theorem \ref{th:NA} to the specific function spaces and boundary integral operator relevant to the impedance screen problem. 



\begin{thm}\label{th:NA}
Let $V(\Gamma^\dag)$ be a Hilbert space and $V^*(\Gamma^\dag)$ a unitary realisation of its dual space. Let $V(\Gamma)$, $V(\Gamma_j)$, $V_h(\Gamma_j)$, $j\in \N_0$, be closed subspaces of $V(\Gamma^\dag)$ with $V_h(\Gamma_j)\subset V(\Gamma_j)$, $j\in \N_0$. Suppose that 
\begin{enumerate}[(i)]
\item\label{ATcoercive} $A:V(\Gamma^\dag) \to V^*(\Gamma^\dag)$ is compactly perturbed coercive; 
\item\label{Ainvertible} $A$ is invertible on $V(\Gamma)$; 
\item\label{Mosco_j} $V(\Gamma_j)\xrightarrow M V(\Gamma)$ as $j\rightarrow\infty$; (Mosco convergence)
\item\label{th:NA:ii} there exists a dense subspace $\tilde{W}\subset V(\Gamma)$ such that, for all $w\in \tilde{W}$,\[\inf_{v_h\in V_h(\Gamma_j)}\|w-v_h\|_{V(\Gamma^\dag)}\to0, \quad j\to\infty.\]
	\end{enumerate}
Then there exists $J\in\N$ such that for each $j\geq J$ and $f\in V^*(\Gamma^\dag)$ the problem: find $u^h_j\in V_h(\Gamma_j)$ such that
\begin{equation} \label{eq:varVhj}
\langle A u^h_j,v^h_j\rangle = \langle f,v^h_j\rangle_{V^*(\Gamma^\dag)\times V(\Gamma^\dag)}, \quad \mbox{for all } v^h_j\in V_h(\Gamma_j),
\end{equation}
is well-posed, and $\|u^h_j-u\|_{V(\Gamma^\dag)}\to 0$ as $j\to\infty$, where $u$ is the unique element of $V(\Gamma)$ satisfying
\begin{equation} \label{eq:varV}
\langle A u,v\rangle = \langle f,v\rangle_{V^*(\Gamma^\dag)\times V(\Gamma)}, \quad \mbox{for all } v\in V(\Gamma).
\end{equation}
\end{thm}


\begin{proof}
This follows from Lemma \ref{lem:dec3} with $H=V(\Gamma^\dag)$, $W=V(\Gamma)$ and $W_j=V_h(\Gamma_j)$; the fact that $V_h(\Gamma_j)\xrightarrow M V(\Gamma)$ follows from \eqref{Mosco_j}, \eqref{th:NA:ii} and \cite[Lemma~2.4]{chandler-wilde2019}.
\end{proof}




Before considering BEM convergence we first apply Theorem \ref{th:NA} to prove convergence of the exact BIE solutions on prefractals (i.e.\ with no discretization).
\begin{thm}[BIE convergence for the impedance screen problem]\label{thm:impedanceconv}
Let $\Gamma$ and $\Gamma_j$, $j\in \N_0$, be bounded open subsets of $\Gamma_\infty=\R^{n-1}\times\{0\}$, with $\Gamma\subset \Gamma^\dag$ and $\Gamma_j\subset \Gamma^\dag$, $j\in \N_0$, for some bounded open $\Gamma^\dag\subset\Gamma_\infty$. 
For $\Omega$ denoting any of $\Gamma$, $\Gamma^\dag$ or $\Gamma_j$, $j\in\N_0$, define $V(\Omega):=\tilde{H}^{1/2}(\Omega)\times \tilde{H}^0(\Omega)$ (recall that $\tilde{H}^0(\Omega)\cong L^2(\Omega)$). 
Let $A$ denote the operator introduced in Theorem \ref{BVPBIEequiv}, with $\Gamma$ replaced by $\Gamma^\dag$, and suppose that 
\begin{align}
\label{MoscoPrefract}
\tilde{H}^{1/2}(\Gamma_j)\xrightarrow M \tilde{H}^{1/2}(\Gamma) \text{ and }\tilde{H}^0(\Gamma_j)\xrightarrow M \tilde{H}^0(\Gamma) 
\text{ as }j\rightarrow\infty. \qquad \text{(Mosco convergence)}
\end{align}
Let $u_j$ denote the solution of \eqref{eq:varVhj} with $V_h(\Gamma_j)=V(\Gamma_j)$. 
Then 
$u_j$ converges to the solution $u$ of \eqref{eq:varV} as $j\to\infty$. 
\end{thm}
\begin{proof}
We check the conditions of Theorem \ref{th:NA}. Condition \eqref{ATcoercive} follows from Theorem \ref{thm:CoercCompact}, with $\Gamma$ replaced by $\Gamma^\dag$, and condition \eqref{Ainvertible} follows from Theorem \ref{thm:BIEwellposed}. 
Condition \eqref{Mosco_j} is immediate from \eqref{MoscoPrefract}, and condition \eqref{th:NA:ii}  is implied for $\tilde{W}=V(\Gamma)$ by \eqref{Mosco_j}.
\end{proof} 

We now state and prove our BEM convergence result.

\begin{thm}[BEM convergence for the impedance screen problem]\label{thm:impedanceNA}
Let $\Gamma$, $\Gamma_j$, $\Gamma_\dag$, $V(\Omega)$ and $A$ be as in Theorem \ref{thm:impedanceconv}, and assume that \eqref{MoscoPrefract} holds.

In the case $n=2$, assume that $\Gamma_j$ is a finite union of open intervals whose closures are mutually disjoint. In the case $n=3$, assume that $\Gamma_j$ is a finite disjoint union of simple open polygons whose boundaries intersect only at vertices (or not at all). Let $\cT_h(\Gamma_j)$ 
denote a triangulation of $\Gamma_j$ (in the sense of \cite[(3.3.11)]{BrSc:07}) by intervals (in the case $n=2$) or triangles (in the case $n=3$).
Define 
$V_h(\Gamma_j):=V_{h,0}^1(\Gamma_j)\times V_h^0(\Gamma_j)$, where $V_{h,0}^1(\Gamma_j)$ denotes the space of continuous piecewise linear functions on $\cT_h(\Gamma_j)$ which vanish on $\partial\Gamma_j$, and $V_{h}^0(\Gamma_j)$ denotes the space of piecewise constant functions on $\cT_h(\Gamma_j)$, both extended by zero to the whole of $\Gamma_\infty$.

Suppose that
\begin{enumerate}[(i)$\,'$]
\item\label{Jcond} for any compact $K\subset\Gamma$ there exists $J'\in\N$ such that $K\subset\Gamma_j$ for all $j\geq J'$;
\item\label{chunkiness} the meshes are uniformly nondegenerate, i.e.\ there exists $\sigma>0$ such that $\sigma(T)\leq \sigma$ for all $T\in \cT_h(\Gamma_j)$ and $j\in \N_0$, where $\sigma(T)$ is the chunkiness parameter in \cite[\S3]{BrSc:07};
\item\label{h_jcond} $h_j:=\max_{T\in \cT_h(\Gamma_j)}\{\diam(T)\}\to0$ as $j\to\infty$.
\end{enumerate}
Then the BEM solution $u^h_j$ of \eqref{eq:varVhj} converges to the solution $u$ of \eqref{eq:varV} as $j\to\infty$. In this case we say that ``BEM convergence holds''. 
\end{thm}

\begin{proof}
We check the conditions of Theorem \ref{th:NA}. Conditions \eqref{ATcoercive}-\eqref{Mosco_j} were verified already in the proof of Theorem \ref{thm:impedanceconv}. For condition \eqref{th:NA:ii} we take $\tilde{W}:=\left(C_0^\infty(\Gamma)\right)^2$, which is dense in $V(\Gamma)$ by definition of the latter. By \eqref{Jcond}$'$, for any $w\in \tilde{W}$ there exists $J'\in\N$ such that $\supp{w}\subset \Gamma_j$, and hence $w\in (C^\infty_0(\Gamma_j))^2$, for all $j\geq J'$. 
Let $w=(w_1,w_2)$, with $w_1,w_2\in C^\infty_0(\Gamma_j)$, and in the sequel let $C$ denote an arbitrary constant independent of $w$ and $h_j$. By standard piecewise polynomial approximation results (e.g.\ \cite[Theorem~4.4.4 and Theorem 4.4.20]{BrSc:07}) and the nondegeneracy condition \eqref{chunkiness}$'$, the interpolatory projection $\cI_h w_1$ of $w_1$ onto $V^1_{h,0}(\Gamma_j)\subset \tH^1(\Gamma_j)\subset H^1(\Gamma_\infty)$ satisfies 
		\begin{align*}
		\|w_1-\cI_h w_1\|_{H^1(\Gamma_\infty)}&=\|{w_1}|_{{\Gamma_j}}-{\cI_h w_1}|_{\Gamma_j}\|_{W^1({\Gamma_j})}\leq Ch_j|w_1|_\Gamma|_{W^2({\Gamma_j})}=Ch_j|w_1|_{W^2({\Gamma_\infty})},\\
		\|w_1-\cI_h w_1\|_{L^2(\Gamma_\infty)}&=\|{w_1}|_{{\Gamma_j}}-{\cI_h w_1}|_{\Gamma_j}\|_{L^2({\Gamma_j})}\leq Ch_j^2|w_1|_\Gamma|_{W^2({\Gamma_j})}=Ch_j^2|w_1|_{W^2({\Gamma_\infty})},
		\end{align*}
where $|\cdot|_{W^2({\Gamma_\infty})}$ is the usual $W^2$ seminorm, 
so that by function space interpolation
\begin{align}\label{eq:split_bound3}
\|w_1-\cI_h w_1\|_{\tH^{1/2}(\Gamma_j)} = \|w_1-\cI_h w_1\|_{H^{1/2}(\Gamma_\infty)}\leq Ch_j^{3/2}|w_1|_{W^2({\Gamma_\infty})}.
		\end{align}
Similarly (see e.g.\ \cite[Lemma A.1]{chandler-wilde2019}) the $L^2$ projection $\Pi_h w_2$ of $w_2$ onto $V_{h}^0(\Gamma_j)$ satisfies
\begin{align}
\label{eq:split_bound4}
		\|w_2-\Pi_h w_2\|_{\tH^0(\Gamma_j)} = \|w_2-\Pi_h w_2\|_{L^2(\Gamma_\infty)} \leq Ch_j\left|w_2\right|_{W^1(\Gamma_\infty)},
\end{align}
and combining \eqref{eq:split_bound3} and \eqref{eq:split_bound4} gives
\begin{align}
	\inf_{v_h\in V_{h}(\Gamma_j)}\|w-v_h\|_{V(\Gamma^\dag)} &\leq
\|w_1-\cI_h w_1\|_{\tH^{1/2}(\Gamma_j)}+ \|w_2-\Pi_h w_2\|_{\tH^0(\Gamma_j)}  \notag\\
&\leq C \left(h_j^{3/2} \left|w_1\right|_{W^2(\Gamma_\infty)} + h_j \left|w_2\right|_{W^1(\Gamma_\infty)}\right)\notag\\
&\leq C h_j \|w\|_{(W^2(\Gamma_\infty))^2},
\label{eqn:bestapprox}
\end{align}
where in bounding $h_j^{3/2}\leq Ch_j$ we invoked 
\eqref{h_jcond}$'$. 
Moreover, since $h_j$ is the only $j$-dependent factor on the right-hand side of \eqref{eqn:bestapprox}, condition \eqref{th:NA:ii} follows by \eqref{h_jcond}$'$. 
\end{proof}

\begin{rem}
Note that condition \eqref{Jcond}$\,'$ is not implied by \eqref{MoscoPrefract}, as the following example shows. Let $n=3$, let $\Gamma:=(-1,1)^2$, and for each $j\in \N$ let $\Gamma_j:=(-1,1)^2\setminus [-1/(2j),1/(2j)]^2$. Then \eqref{MoscoPrefract} holds, since \eqref{MoscoPrefract} holds with $\Gamma$ replaced by $\Gamma':=(-1,1)^2\setminus\{(0,0)\}$ by \cite[Prop.~4.3(i)]{chandler-wilde2019} (cf.\ the proof of Theorem \ref{prop:Koch} below), and $\tH^{1/2}(\Gamma) =\tH^{1/2}(\Gamma')$ and $\tH^{0}(\Gamma) =\tH^{0}(\Gamma')$ by \cite[Thm~3.12 and Lem.~3.10(xi)]{ChaHewMoi:13}. However, \eqref{Jcond}$\,'$ fails for any $K$ containing $(0,0)$.
\end{rem}

\begin{rem}
\label{rem:CantorDust}
While our main focus is on the case where the screen $\Gamma$ is bounded and open with fractal boundary, Theorem \ref{th:NA} also allows us to comment on the case where $\Gamma$ is a compact fractal with empty (relative) interior, e.g.\ a Cantor set (for $n=2$), Cantor dust (for $n=3$), or Sierpinski triangle (for $n=3$). It turns out that the scattered field for such screens is zero whenever $\Gamma$ has zero $(n-1)$-dimensional Lebesgue measure (which is true for all three examples just given). 

In more detail, let $\Gamma\subset\Gamma_\infty\cong \R^{n-1}$ be non-empty and compact, and let $(\Gamma_j)$ be a sequence of bounded non-empty open subsets of $\Gamma_\infty$ such that 
$\Gamma=\cap_{j\in\N} \overline{\Gamma_j}$,  $\overline{\Gamma_{j+1}}\subset\overline{\Gamma_j}$ and $\tH^{s}(\Gamma_j)= H^{s}_{\overline{\Gamma_j}}$, $j\in\N$ (for example, $\Gamma_j$ could be the standard prefractal approximations to one of the three examples given above, cf.\ \cite[\S\S6.1-6.3]{chandler-wilde2019} and Figure \ref{fig:CantorDustNumerics}(a) below). Then by \cite[Prop.~4.3(ii)]{chandler-wilde2019}, which is stated for $s=-1/2$ but whose proof works also for $s\in \R$, we have $\tH^{s}(\Gamma_j)\xrightarrow M H^{s}_{\Gamma}$, $s\in\R$. 
Hence if $\Gamma$ has zero $(n-1)$-dimensional Lebesgue measure then 
$H^{1/2}_{\Gamma}=H^{0}_{\Gamma}=\{0\}$, so that, arguing as in the proof of Theorem \ref{thm:impedanceconv}, the solution of \eqref{eq:varVhj} with $V_h(\Gamma_j)=V(\Gamma_j)$ converges to $(0,0)$ as $j\to\infty$. Consequently, the scattered wave field for the impedance BVP vanishes in the limit as $j\to\infty$. 
Numerical verification of this fact is given for the middle third Cantor dust in \S\ref{subsec:CantorDustNumerics}.
\end{rem}


\section{Examples}\label{sec:BEMegs}

In this section we apply our convergence theory to some specific examples of fractal screens. 

\subsection{Classical Snowflakes}
\label{sec:ClassicalSnowflakes}

\begin{figure}
\includegraphics[width=\linewidth]{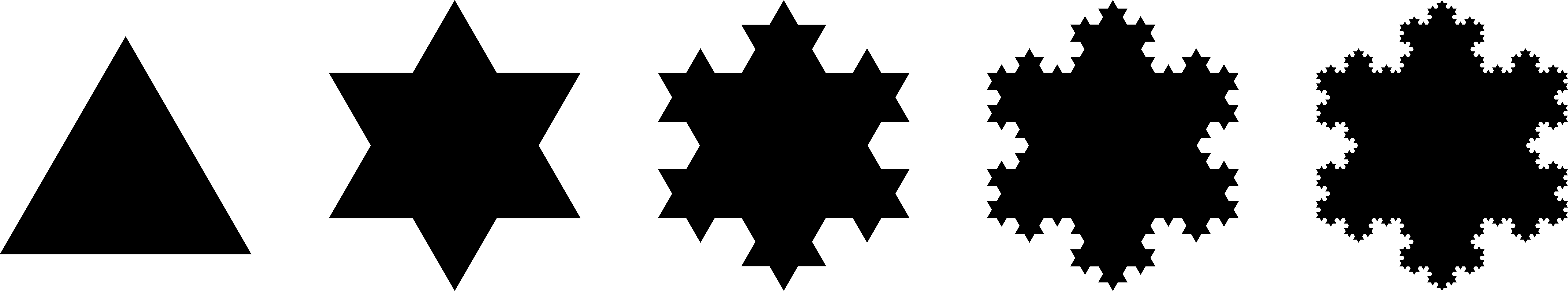}

\vspace{2mm}

\includegraphics[width=\linewidth]{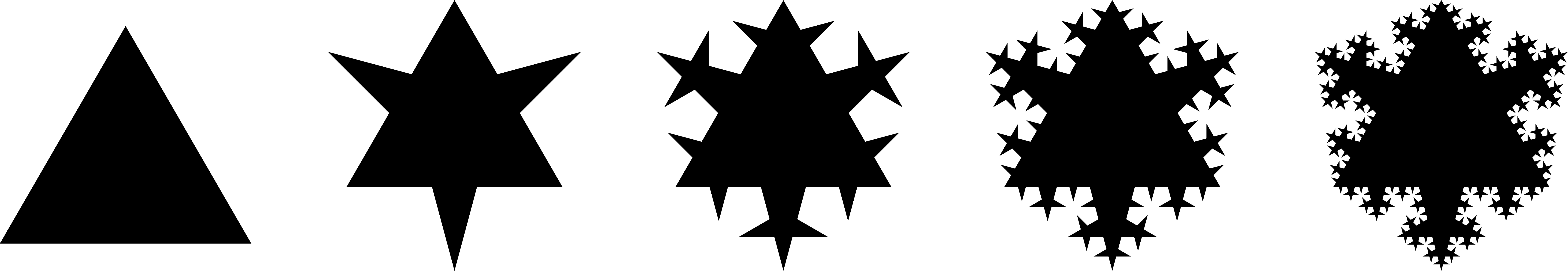}
\caption{The first 5 prefractals $\Gamma_0,\ldots,\Gamma_4$ for the classical snowflakes with $\beta=\pi/6$ (top row) and $\beta=\pi/12$ (bottom row).\label{fig:Snowflakes}}
\end{figure}

The classical snowflakes are a family of bounded open plane sets with fractal boundary, generalising the Koch snowflake. Fix $0<\beta<\pi/2$. 
Let $\Gamma_0$ be the open unit equilateral triangle, and for $j\geq 1$ define $\Gamma_{j}$ iteratively 
by adding open isoceles triangles of apex angle $2\beta$ and leg length $(2(1+\sin\beta))^{-j}$, together with the relative interiors of their bases, at the midpoints of each of the sides of $\Gamma_{j-1}$, as illustrated in Figure~\ref{fig:Snowflakes} for the cases $\beta=\pi/6$, which corresponds to the Koch snowflake, and $\beta=\pi/12$. 
This generates a nested increasing sequence of prefractals satisfying $\Gamma_{j-1}\subset\Gamma_j$, $j\in\N$, from which we define $\Gamma:=\cup_{j\in\N_0}\Gamma_j$. (For full details of the construction see \cite[\S6.4]{chandler-wilde2019}.) 

\begin{prop}
\label{prop:Koch}
Let $0<\beta<\pi/2$, and let $\Gamma$ denote the classical snowflake and $\Gamma_j$, $j\in\N_0$ its prefractal approximations defined above. Let the BEM spaces $V_h(\Gamma_j)$ be defined as in Theorem \ref{thm:impedanceNA}, and suppose that the associated triangulations are uniformly nondegenerate, with $h_j\to0$ as $j\to\infty$. Then BEM convergence holds for the impedance screen problem. 
\end{prop}
\begin{proof}
We just need to check conditions \eqref{MoscoPrefract}$'$ and \eqref{Jcond}$'$ in Theorem \ref{thm:impedanceNA}. 
Condition \eqref{Jcond}$'$  is a simple consequence of the nestedness of $\Gamma_j$ and the fact that $\Gamma=\cup_{j\in\N_0}\Gamma_j$. But by \cite[Prop.~4.3(i)]{chandler-wilde2019}, which is stated for $s=-1/2$ but whose proof works also for $s\in \R$, these two facts also imply \eqref{MoscoPrefract}$'$.
\end{proof}

\subsection{Square Snowflake}
\label{sec:SquareSnowflake}

\begin{figure}
\includegraphics[width=\linewidth]{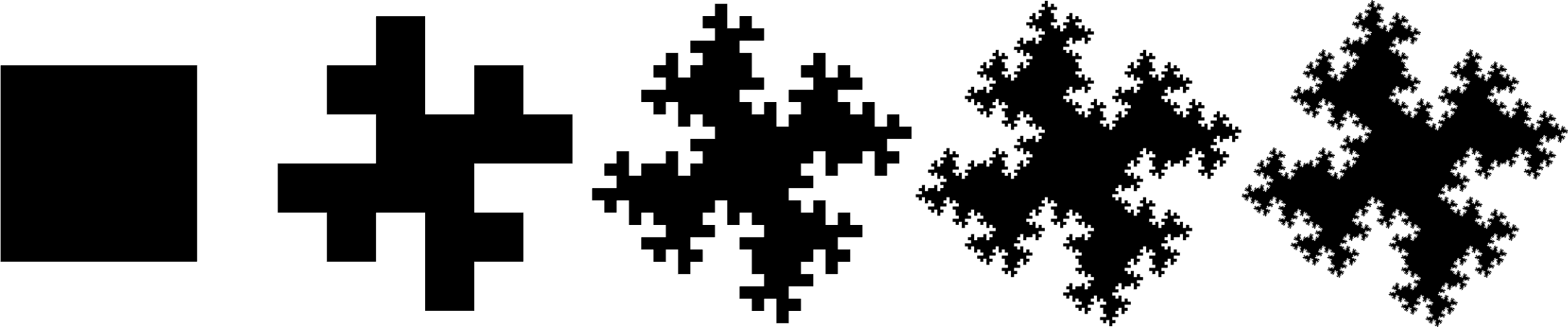}
\caption{The first 5 prefractals $\Gamma_0,\ldots,\Gamma_4$ for the square snowflake. 
\label{fig:SquareSnowflakes}}
\end{figure}

The square snowflake is another bounded open plane set with fractal boundary, the first five prefractals of which are shown in Figure~\ref{fig:SquareSnowflakes}. Here $\Gamma_0$ is the open unit square, and for $j\geq 1$ the prefractal $\Gamma_{j}$ is constructed by removing and adding a square of side-length $1/4^j$ on each of the sides of $\Gamma_{j-1}$. (For details see \cite[\S5.2]{caetano2018}.) 
The resulting sequence of prefractals does not enjoy the nestedness property of the prefractals for the classical snowflakes. However, it was proved in \cite[\S5.2]{caetano2018} that there exists a nested increasing sequence of open sets $(\Gamma_j^-)_{j\in\N_0}$ and a nested decreasing sequence of compact sets $(\Gamma_j^+)_{j\in\N_0}$ such that 
\begin{align}
\label{eqn:inclusions}
\Gamma^-_{j}\subset\Gamma_j\subset\Gamma^+_{j},\qquad j\in \N_0,
\end{align}
and
\begin{align}
\label{eqn:closure}
\overline{\bigcup_{j\in\N_0} \Gamma_j^-} = \bigcap_{j\in \N_0}\Gamma_j^+.
\end{align} 
One can then define $\Gamma = \cup_{j\in\N_0}\Gamma^-_j$.
Furthermore, it was also shown in \cite[\S5.2]{caetano2018} that $\Gamma$ is a so-called \emph{thick} domain (in the sense of Triebel, cf.\ \cite[\S3]{Tri08}), and hence that 
\begin{align}
\label{eqn:tildesubscript}
\tH^s(\Gamma)=H^s_{\overline\Gamma}, \qquad s\in\R.
\end{align}
\begin{prop}
\label{prop:Square}
Let $\Gamma$ denote the square snowflake and $\Gamma_j$, $j\in\N_0$ its prefractal approximations defined above. Let the BEM spaces $V_h(\Gamma_j)$ be defined as in Theorem \ref{thm:impedanceNA}, and suppose that the associated triangulations are uniformly nondegenerate, with $h_j\to0$ as $j\to\infty$. Then BEM convergence holds for the impedance screen problem. 
\end{prop}
\begin{proof}
Again, we just need to check conditions \eqref{MoscoPrefract}$'$ and \eqref{Jcond}$'$ in Theorem \ref{thm:impedanceNA}. 
Condition \eqref{Jcond}$'$  is a simple consequence of the nestedness of $\Gamma_j^-$, the first inclusion in \eqref{eqn:inclusions}, and the fact that $\Gamma=\cup_{j\in\N_0}\Gamma_j^-$. But, combined with \eqref{eqn:closure}, \eqref{eqn:tildesubscript} and \cite[Prop.~4.3(iii)]{chandler-wilde2019} (again, generalised from $s=-1/2$ to $s\in \R$), these facts also imply \eqref{MoscoPrefract}$'$.

\end{proof}


\section{Numerical experiments}
\label{sec:Numerical}

In this section we present numerical experiments validating the theoretical results of \S\ref{sec:BEM} and \S\ref{sec:BEMegs}. For brevity we consider only the case $n=3$, which corresponds to scattering by a 2D planar screen in 3D ambient space. To implement the BEM approximation spaces described in Theorem \ref{thm:impedanceNA} we use the open-source software package Bempp \cite{SBAPS15}, which is downloadable from \verb|bempp.com|. 

An obvious challenge in the implementation of BEMs for scattering by fractal screens is that, for examples like those in \S\ref{sec:BEMegs}, as $j$ increases the geometric complexity of the boundary $\partial\Gamma_j$ of the prefractal $\Gamma_j$ increases significantly. This makes it non-trivial to construct a suitable sequence of uniformly non-degenerate meshes, and, moreover, the number of mesh elements required can grow rapidly with increasing $j$. 
The development of efficient meshing strategies for general sequences of prefractals is the focus of ongoing research by the authors, the results of which will be presented in future publications. However, in the present paper we restrict our attention to the simplest possible situation in which one can use a uniform mesh. 

More precisely, we henceforth make the assumption that each prefractal $\Gamma_j$ can be embedded in a uniformly-meshed parallelogram, in the sense explained in \S\ref{sec:BEMimplementation}. This endows the BEM matrix with special structure that can be exploited to allow efficient storage and fast inversion using an iterative method (GMRES in our case) with matrix-vector products computed using the Fast Fourier Transform (FFT).  This assumption is satisfied by the standard Koch snowflake ($\beta=\pi/6$) of \S\ref{sec:ClassicalSnowflakes} and the square snowflake of \S\ref{sec:SquareSnowflake}, numerical results for both of which will be reported in \S\ref{sec:NumericalResults}. 

Our FFT-based approach will allow us to present results for larger problems (specifically, on higher order prefractals) than were considered for the Dirichlet case in \cite{chandler-wilde2019} (where uniform meshes were also used), in spite of the increase in system size associated with the need to discretize two unknowns in the impedance case rather than just one in the Dirichlet case. 
However, 
we expect that to approximate the solutions most efficiently (in terms of the size of the approximation space), and to deal with more general $\Gamma$ (such as the snowflakes with $\beta\neq \pi/6$ in \S\ref{sec:ClassicalSnowflakes}), one should use non-uniform meshes, refined towards areas of geometric complexity, constructed either a priori, as per the mesh used in \cite{Bagnerini13} in the context of heat conduction across fractal interfaces, or adaptively as the computation progresses. 
For such meshes FFT-based inversion is not applicable, and one should consider alternative compression/acceleration techniques such as the fast multipole method or $\mathcal{H}$-matrices. 
Furthermore, to develop a time-efficient solver one would need to consider a suitable preconditioning strategy to accelerate the convergence of the GMRES iteration. (In all the experiments we ran, unpreconditioned GMRES converged within a few hundred iterations, which was acceptable for our purposes.)
However, since the main focus of this paper is on analysis rather than numerics, we leave consideration of these issues for future work.






\subsection{BEM implementation on uniform meshes}
\label{sec:BEMimplementation}

\begin{figure}[t]
	\hspace{-1.1cm}
	\subfloat[][$j=1$, $M_x=M_y=8$, $h_1=1/8$]
	{\includegraphics[width=.55\linewidth]{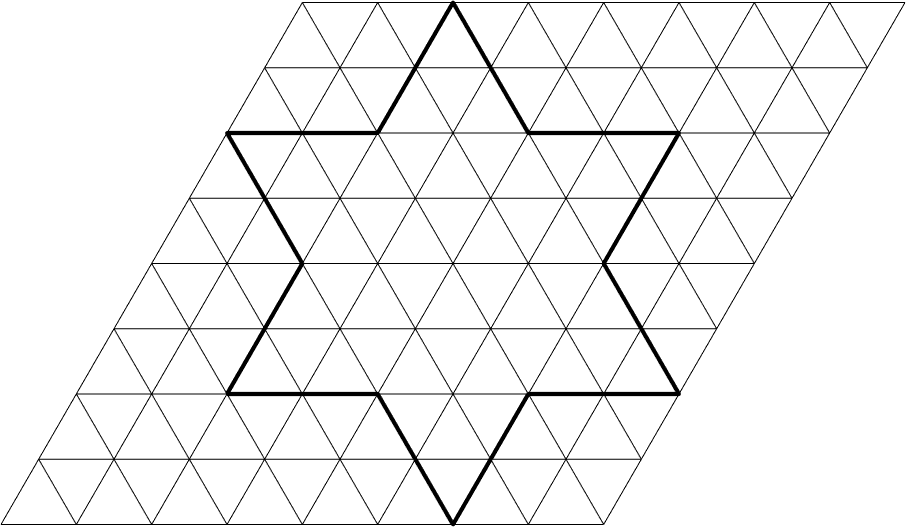}\label{k12}}
	\subfloat[][$j=1$, $M_x=M_y=16$, $h_1=1/16$]
	{\includegraphics[width=.55\linewidth]{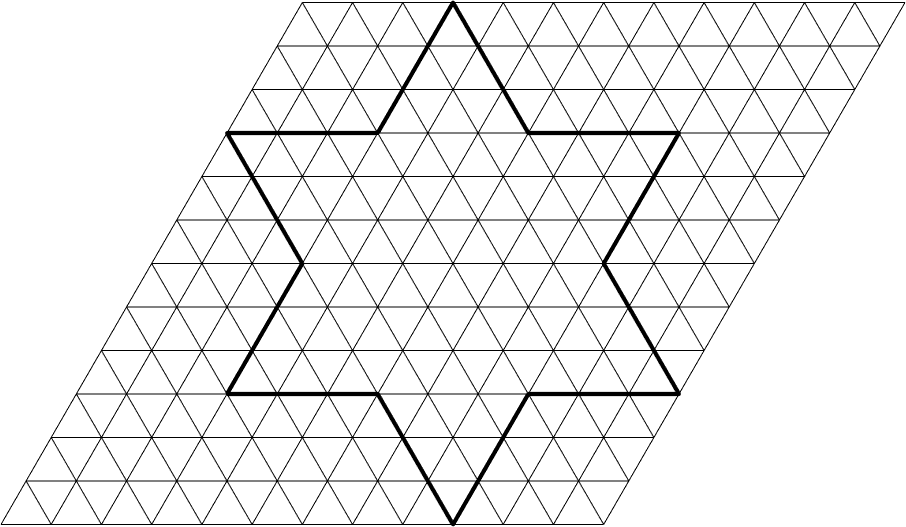}\label{k13}}	
	\\

	\hspace{-1.1cm}
	\subfloat[][$j=2$, $M_x=M_y=16$, $h_2=1/16$]
	{\includegraphics[width=.55\linewidth]{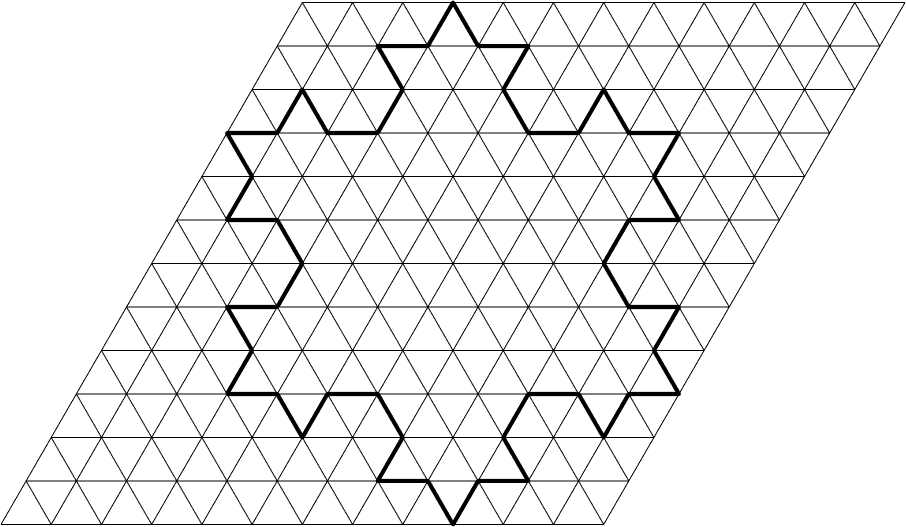}\label{k21}}
	\subfloat[][$j=2$, $M_x=M_y=32$, $h_2=1/32$]
	{\includegraphics[width=.55\linewidth]{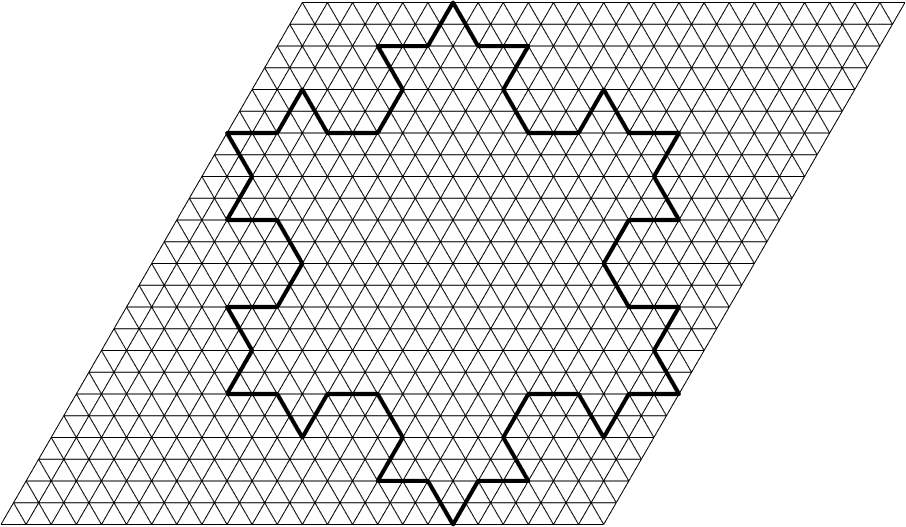}\label{k22}}
	\caption{Prefractals for the Koch snowflake of \S\ref{sec:ClassicalSnowflakes}, embedded in uniformly meshed parallelograms with $\theta=\pi/3$ and $R_x=R_y=1$. Subfigures \protect\subref{k12} and \protect\subref{k13} show $\Gamma_1$, and \protect\subref{k21} and \protect\subref{k22} show $\Gamma_2$, each with two different mesh widths. 
}\label{fig:Koch_eg}
\end{figure}

\begin{figure}[t!]
	\hspace{.8cm}
	\subfloat[][
	$j=1$, $R_x=R_y=1.5$, $M_x=M_y=6$, $h_1=1/4$]
	{\includegraphics[trim = 250 100 100 100, clip,width=.5\linewidth]{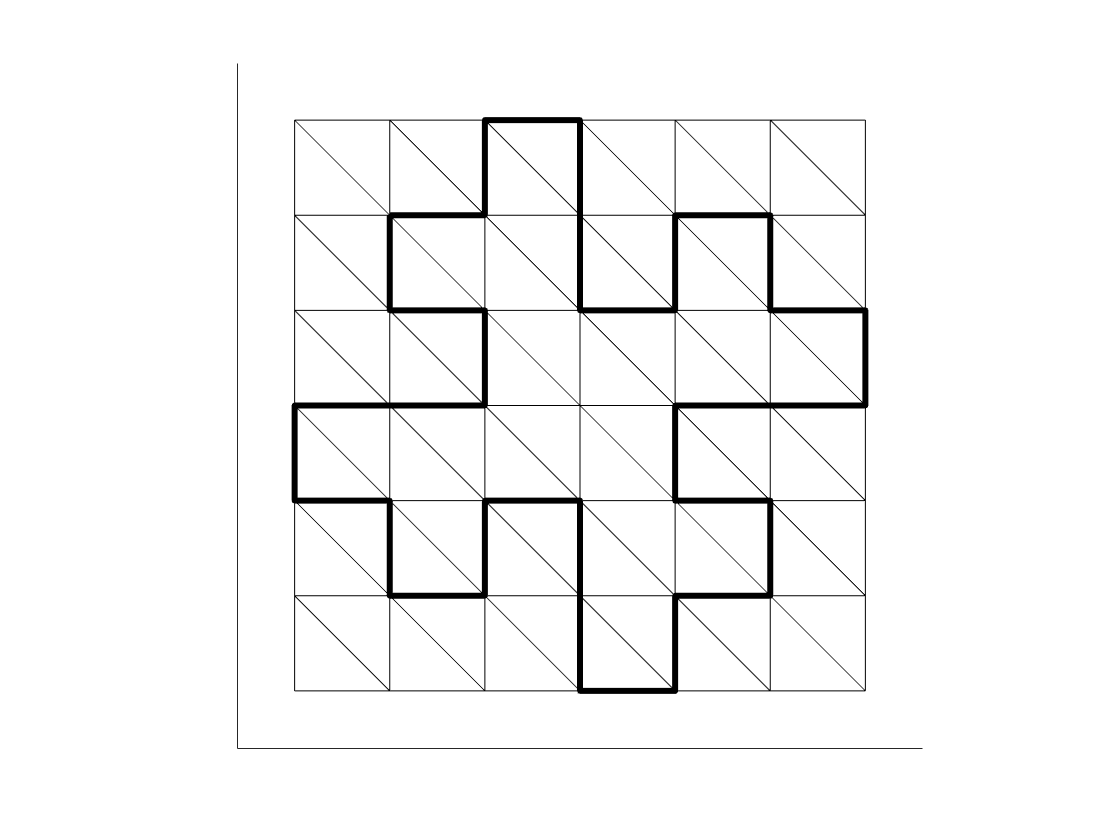}\label{s12}}
	\subfloat[][
		$j=1$, $R_x=R_y=1.5$, $M_x=M_y=12$, $h_1=1/8$]
		{\includegraphics[trim = 250 100 100 100, clip,width=.5\linewidth]{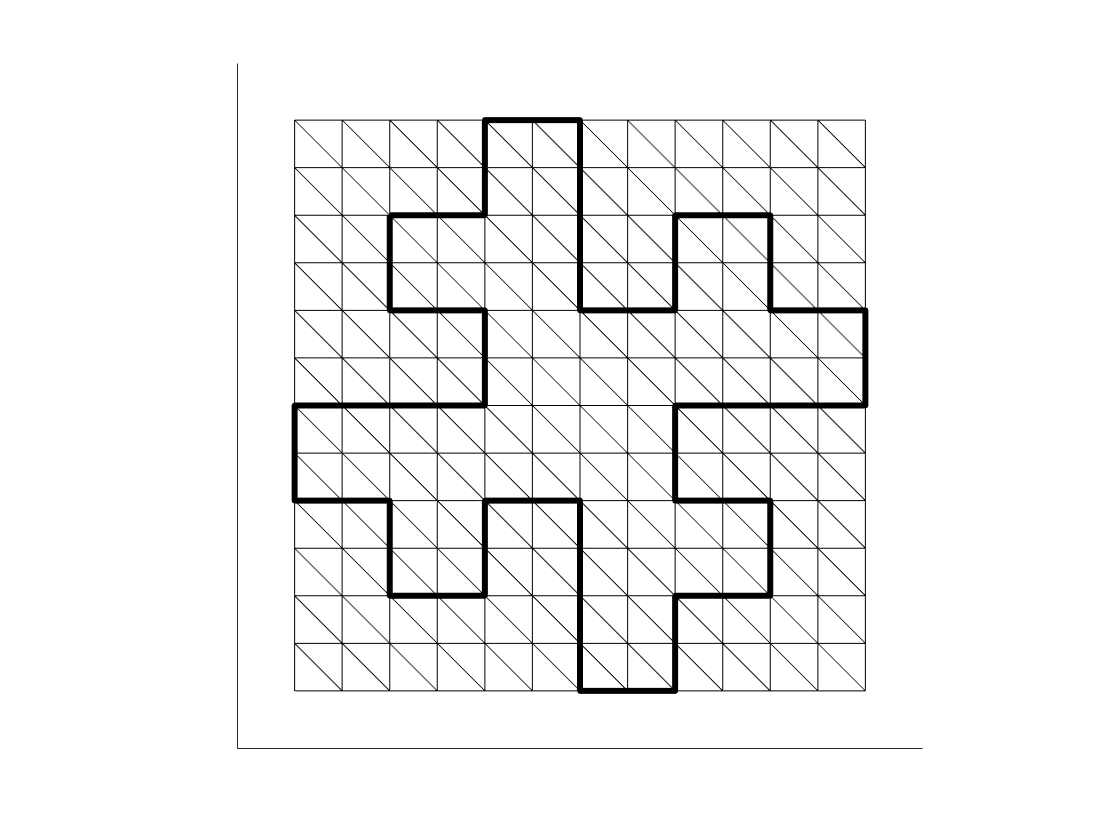}\label{s13}}\\
	
	\hspace{.8cm}
	\subfloat[][
		$j=2$, $R_x=R_y=1.625$, $M_x=M_y=26$, $h_1=1/16$]
		{\includegraphics[trim = 250 100 100 90, clip,width=.5\linewidth]{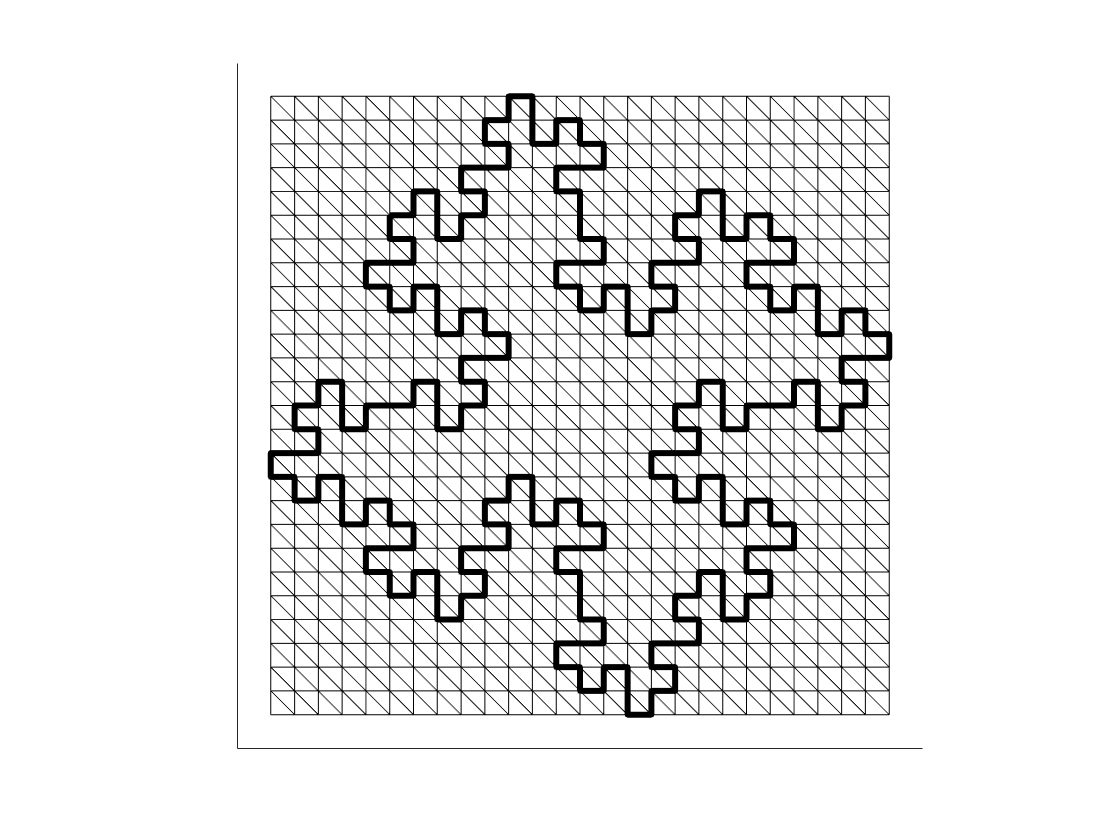}\label{s21}}
	\subfloat[][
			$j=2$, $R_x=R_y=1.625$, $M_x=M_y=52$, $h_1=1/32$]
	{\includegraphics[trim = 250 100 100 90, clip,width=.5\linewidth]{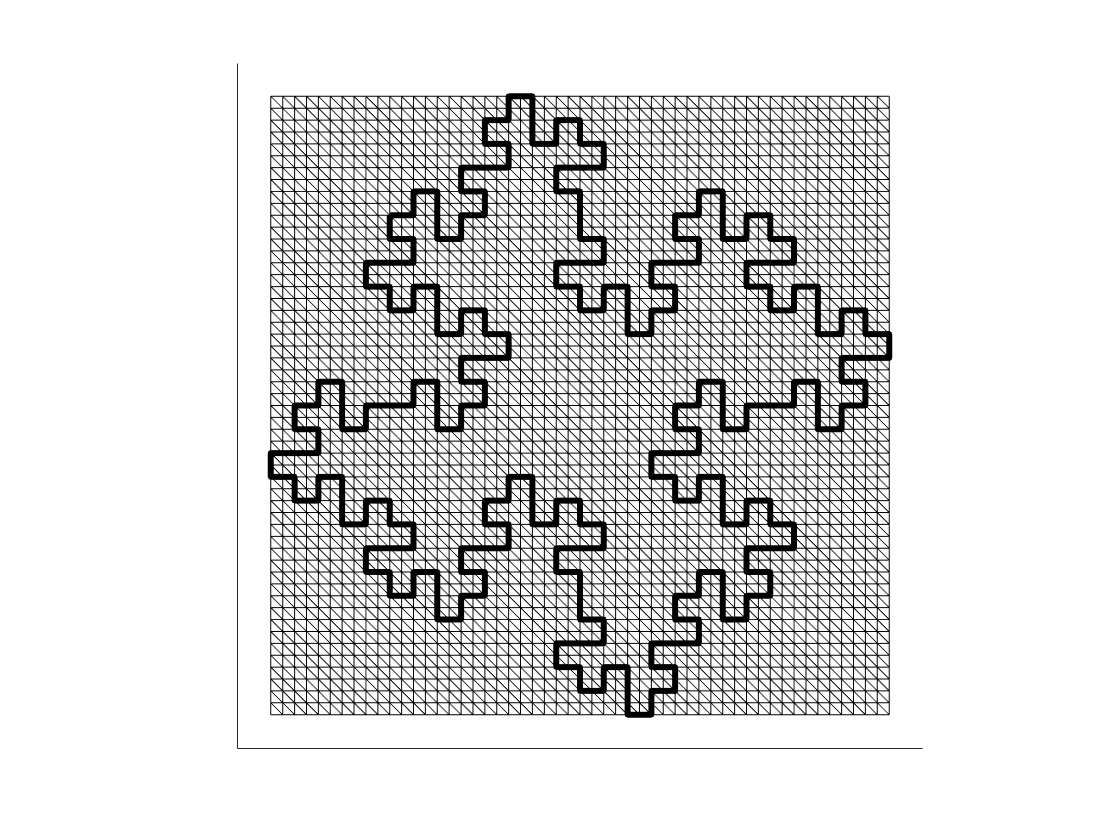}\label{s22}}
		\caption{Analogue of Figure \ref{fig:Koch_eg} for the square snowflake of \S\ref{sec:SquareSnowflake}. Here $\theta=\pi/2$. 
	}\label{fig:square_eg}
\end{figure}

For each $j\in\N$, we assume that 
there exists a parallelogram
\[ \cP_j = \{(x,y)\in \R^2:0\leq y< R_y\sin\theta,\,\frac{y}{\tan{\theta}}<x<R_x+\frac{y}{\tan{\theta}}\},\]
with side lengths $R_x,R_y>0$ and acute interior angle $\theta\in (0,\pi/2]$, equipped with a uniform mesh $\cT_h(\cP_j)$, 
such that the prefractal $\Gamma_j$ and its mesh $\cT_h(\Gamma_j)$ conform to $\cT_h(\cP_j)$ in the sense that 
\[
\cT_h(\Gamma_j)\subset \cT_h(\cP_j).\] 
Specifically, the mesh $\cT_h(\cP_j)$ is assumed to be based on a subdivision of the sides of $\cP_j$ into $M_x$ and $M_y$ elements in the horizontal and vertical directions respectively, comprising a family of $M_xM_y$ ``upward-pointing'' triangles
\[ a e_1 + be_2 + {\rm int}\left(\hull(O,e_1,e_2)\right), \qquad a=0,\ldots,M_x-1, \,b=0,\ldots,M_y-1 \]
and a family of $M_xM_y$ ``downward-pointing'' triangles
\[ a e_1 + be_2 + {\rm int}\left(\hull(e_1,e_2,e_1+e_2)\right), \qquad a=0,\ldots,M_x-1, \,b=0,\ldots,M_y-1, \]
where $O=(0,0)$, $e_1=(R_x/M_x,0)$, and $e_2=((R_y/M_y)\cos{\theta},(R_y/M_y)\sin{\theta})$.
Illustrations of this arrangement for the Koch snowflake and square snowflake are provided in Figures \ref{fig:Koch_eg} and \ref{fig:square_eg}.

Under these assumptions, it is natural to choose the basis of $V_h(\Gamma_j)$ (assumed to have dimension $N_j\in \N$) to be a subset of the basis of $V_h(\cP_j)$ (assumed to have dimension $\tilde N_j\in \N$), since then multiplication by the Galerkin matrix $\mathbf{A}$ on $\Gamma_j$ (associated with the discrete problem \eqref{eq:varVhj}) can be expressed in terms of multiplication by the Galerkin matrix $\widetilde{\mathbf{A}}$ on the larger screen $\cP_j$ by
\begin{equation}\label{eq:ToeBodge}
\bA v = \bB^T\tilde{\bA}\bB v,\qquad v\in\C^{N_j},
\end{equation}
where $\bB$ is a sparse $\tilde N_j \times N_j$ matrix with a single non-zero entry in each column; explicitly, $\mathbf{B}_{pq}=1$ if the $p$th basis function of $V_h(\cP_j)$ coincides with the $q$th basis function of $V_h(\Gamma_j)$, and $\mathbf{B}_{pq}=0$ otherwise. 
The attraction of \eqref{eq:ToeBodge} is that, while $\tilde{\mathbf{A}}$ is in general larger than $\mathbf{A}$ (i.e.\ $\tilde N_j>N_j$), multiplication by $\tilde{\mathbf{A}}$ can be carried out cheaply because of the special structure of $\tilde{\mathbf{A}}$ that arises under certain indexing conventions, as we now elucidate. 

\newcommand{\twoVec}[2]{\left({	\begin{array}{c}
			{#1}\\
			{#2}
	\end{array}}\right)}

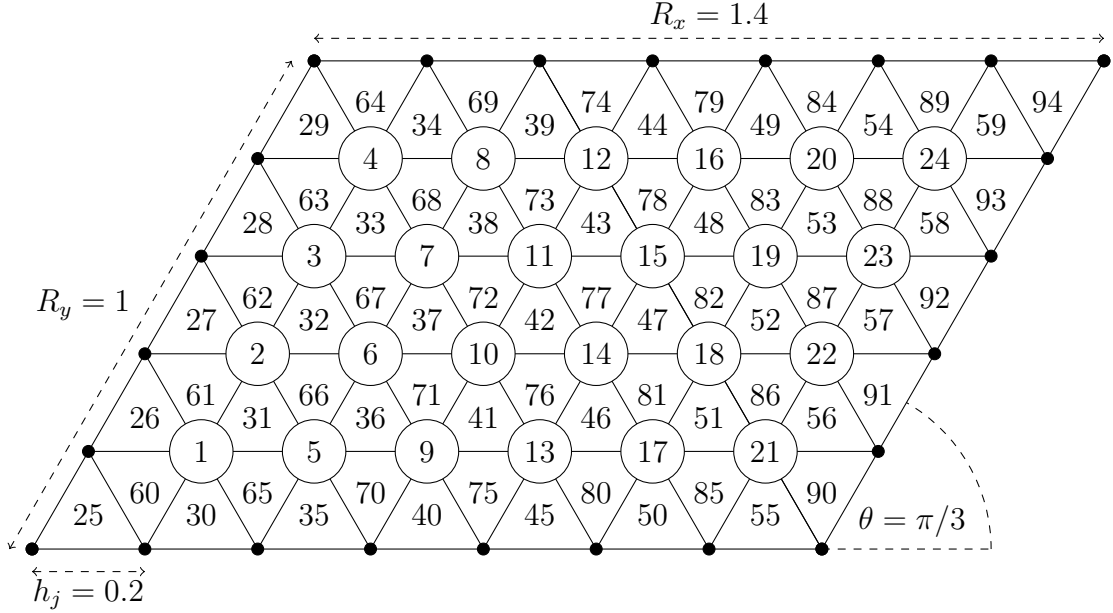
\begin{figure}
	\centering
	\begin{tikzpicture}[ scale=1.5]
	\def\a{.5}
	\def\b{.866}
	\def\Nx{7}
	\def\Ny{5}
	\edef\Nxm{\number\numexpr\Nx-1\relax}
	\edef\Nym{\number\numexpr\Ny-1\relax}
	\edef\NxmNy{\number\numexpr\Nx-\Ny\relax}
	\edef\h{\number\numexpr1/\Nx\relax}
	
	\foreach \y in {0,\Ny}
	\foreach \x in {0,...,\Nx}
	\filldraw (\x+\a*\y,\b*\y) circle (1.5pt);

	\foreach \y in {0,...,\Ny}
	\foreach \x in {0,\Nx}
	\filldraw (\x+\a*\y,\b*\y) circle (1.5pt);
	
	\foreach \x in {0,...,\Nx}
	\draw (\x,0) -- (\x+\a*\Ny,\b*\Ny);
	
	\foreach \y in {0,...,\Ny}
	\draw (\a*\y,\b*\y) -- (\Nx+\a*\y, \b*\y);
	
	\foreach \nn in {1,...,\Ny}
	\draw (\a*\nn,\b*\nn)-- (\nn,0);
	
	\foreach \nn in {2,...,\Nx}
	\draw (\nn+\a*\Ny,\b*\Ny)-- (\Nx+\a*\nn-2*\a,\b*\nn-2*\b);
	
	\foreach \nn in {1,...,\NxmNy}
	\draw (\Ny*\a + \nn,\Ny*\b) -- (\Nx+\nn-2,0);
	
	\foreach \y in {0,...,\Nym}
	\foreach \x in {0,...,\Nxm}
	\edef\index{\number\numexpr\x*\Ny+\y+1+(\Nx-1)*(\Ny-1)\relax}
	\filldraw node at (\x+\a*\y+.5,\b*\y+.3) {\index};
	
	\foreach \y in {0,...,\Nym}
	\foreach \x in {0,...,\Nxm}
	\edef\index{\number\numexpr\x*\Ny+\y+1+\Nx*\Ny+(\Nx-1)*(\Ny-1)\relax}
	\filldraw node at (\x+1+\a*\y,\b*\y+.5) {\index};
	
	\foreach \y in {1,...,\Nym}
	\foreach \x in {1,...,\Nxm}
	{\edef\index{\number\numexpr(\x-1)*\Nym+\y\relax}
		\filldraw[black,fill=white] (\x+\a*\y,\b*\y) circle (8pt);
		\filldraw node at (\x+\a*\y,\b*\y) {\index};}

	\draw[<->,dashed] (0,-.2) -- (1,-.2);
	\filldraw node at (.5,-.4) {$h_j=0.2$};
	
		\draw[<->,dashed] (-.2,0) -- (-.2+\Ny*\a,\Ny*\b);
		\filldraw node at (.5*\Ny*\a-.8,.5*\Ny*\b) {$R_y=1$};

	\draw[<->,dashed] (\Ny*\a,\Ny*\b+.2) -- (\Ny*\a+\Nx,\Ny*\b+.2);
	\filldraw node at ({(\Ny*\a+\Nx+\Ny*\a)/2},\Ny*\b+.4) {$R_x=1.4$};
	
	\draw[dashed] (\Nx+1.5,0) arc (0:60:1.5);
	\draw[dashed] (\Nx,0) -- (\Nx+1.5,0);
	\filldraw node at (\Nx + .8,.25) {$\theta=\pi/3$};
	
	\end{tikzpicture}
	\caption{Schematic showing the indexing of the basis elements in $V_h(\cP_j)$ on the mesh $\cT_{h}(\cP_j)$. The interior mesh nodes (which are indexed before the triangular elements in our scheme) are indicated by circles. Here $\theta=\pi/3$, $R_x=1.4$, $R_y=1$, $M_x=7$, $M_y=5$, and $h_j=0.2$.}\label{fig:para_indexing}
\end{figure}

The standard basis functions in $V_h(\cP_j)$ split into three families: the continuous piecewise linear basis functions for $V_{h,0}^1(\cP_j)$, associated with the interior nodes, and two families of piecewise constant basis functions for $V_{h}^0(\cP_j)$, associated with the upwards- and downwards-pointing triangles. Suppose that we index the basis functions according to the schematic in Figure \ref{fig:para_indexing}, indexing first the nodes, then the upward triangles, then the downward triangles. Then the matrix $\tilde{\mathbf{A}}$ has a $3\times 3$ block structure, and, moreover, by the translation-invariance of the sesquilinear form,  
specifically the fact that 
if $\varphi,\psi,\varphi(\cdot+x), \psi(\cdot+x)\in V(\cP_j)$ for some $x$ then 
\begin{equation}\label{eq:slf_translate}
	\left< A\varphi,\psi\right> = \left< A\varphi(\cdot+x),\psi(\cdot+x)\right>,
	\end{equation}
each of the resulting 9 sub-blocks is block-Toeplitz with Toeplitz blocks (BTTB), with the number and size of each of the constituent sub-sub-blocks indicated in Figure \ref{fig:Astar}. 

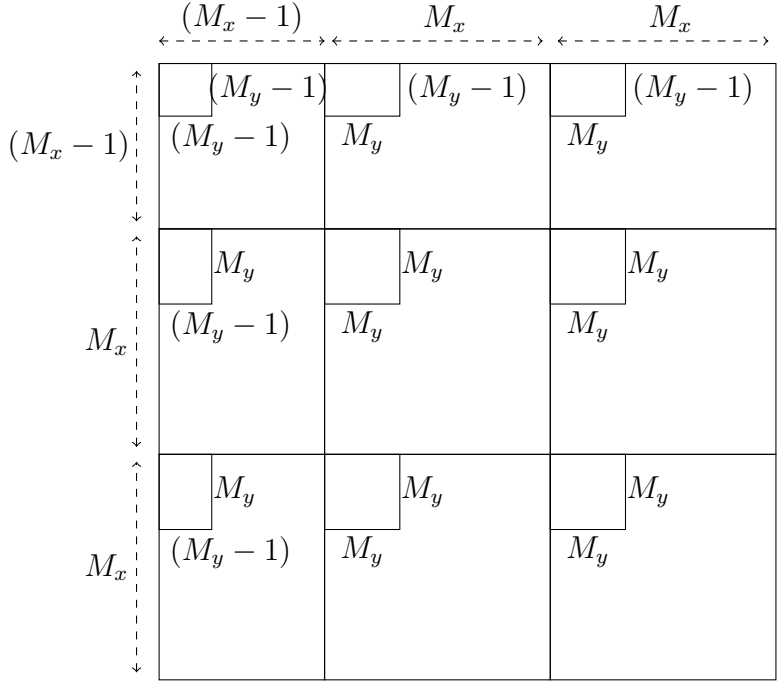
\begin{figure}[t!]
	\centering
	\hspace{-3cm}\begin{tikzpicture}
	\def\W{3}
	\def\Ws{2.2}
	\def\w{1.0}
	\def\ws{.7}
	
	\foreach \x in {0,...,1}
	\foreach \y in {0,...,1}
	{\draw [draw=black] ({\Ws+\W*\x},{\W*(2-\y)}) rectangle ++({\W},{-\W});
		\draw [draw=black] ({\Ws+\W*\x},{\W*(2-\y)}) rectangle++({\w},{-\w});}
	
	{
	\draw [draw=black] (0,2*\W+\Ws) rectangle ++(\Ws,-\Ws);
	\draw [draw=black] (0,2*\W+\Ws) rectangle ++(\ws,-\ws);
}

	\foreach \y in {1,...,2}
		{\draw [draw=black] ({(0)},{(2-\y)*\W}) rectangle++({\Ws},{\W});
		\draw [draw=black] ({(0)},{(3-\y)*\W}) rectangle++({\ws},{-\w});
	}

	\foreach \x in {0,...,1}
	{\draw [draw=black] (\Ws+\W*\x,2*\W+\Ws) rectangle++({\W},{-\Ws});
	\draw [draw=black] (\Ws+\W*\x,2*\W+\Ws)  rectangle++({\w},{-\ws});}
	
	\foreach \x in {1,...,2}
	\foreach \y in {1,...,2}
	{
		\filldraw node at ({\Ws + \W*(\x-1)+.5*\w},{\W*(\y)-\w-.3}) {$M_y$};
		\filldraw node at ({\Ws + \W*(\x-1)+\w+.3},{\W*(\y)-0.5*\w}) {$M_y$};
	}
	
	\foreach \n in {1,...,2}
	{	\draw[<->,dashed] ({\Ws+\W*(\n-1)+.1},{2*\W+\Ws+.3}) -- ({\Ws+\W*\n-.1},{2*\W+\Ws+.3});
		\filldraw node at ({\Ws+\W*(\n-.5)+.1},{2*\W+\Ws+.6} ) {$M_x$};
		\draw[<->,dashed] ({-.3},{\W*(\n-1)+.1}) -- ({-.3},{\W*(\n)-.1});
		\filldraw node at ({-.7},{\W*(\n-.5)}) {$M_x$};
		
		\filldraw node at ({(\n-1)*\W+\Ws+.5*\w},{2*\W+\Ws -\ws-.3}) {$M_y$};
		\filldraw node at ({(\n-1)*\W+\Ws+\w+.9},{2*\W+\Ws -.5*\ws}) {$(M_y-1)$};
		
		\filldraw node at ({\ws+0.3},{(\n)*\W-.5*\w}) {$M_y$};
		\filldraw node at ({.5*\ws+0.6},{(\n)*\W-\w-0.3}) {$(M_y-1)$};
		
	}
	
	\draw[<->,dashed] ({0},{2*\W+\Ws+.3}) -- ({\Ws},{2*\W + \Ws+.3});
	\filldraw node at ({.5*\Ws},{2*\W+\Ws+.6}) {$(M_x-1)$};
	\draw[<->,dashed] ({-.3},{2*\W+.1}) -- ({-.3},{2*\W+\Ws-.1});
	\filldraw node at ({-1.2},{2*\W+.5*\Ws}) {$(M_x-1)$};

	
	\filldraw node at ({.5*\ws+0.6},{2*\W+\Ws-\ws-.3}) {$(M_y-1)$};
	\filldraw node at ({\ws+.75},{2*\W+\Ws-.5*\ws}) {$(M_y-1)$};
	
	\end{tikzpicture}
	\caption{The block structure of the matrix $\tilde{\bA}$ under the indexing convention of Figure \ref{fig:para_indexing}. At the highest level $\tilde{\bA}$ has a $3\times 3$ block structure. Each of the 9 sub-blocks is block-Toeplitz with Toeplitz-blocks (BTTB). The number of sub-sub-blocks in each direction is shown on the \textit{outside} of the figure, while the number of elements in each direction in the sub-sub-blocks is shown on the \textit{inside} of the figure. For example, the (3,3) sub-block has an $M_x \times M_x$ block structure, with each sub-sub-block being an $M_y\times M_y$ matrix.}\label{fig:Astar}
\end{figure}

The BTTB property of the sub-blocks means that the cost of storing the matrix $\tilde{\bA}$ is $O(\tilde N_j)$, 
since $\tilde{\bA}$ can be reconstructed from the first row and column of each sub-sub-block in the first block-row and block-column of each sub-block. Furthermore, for BTTB matrices, cheap matrix-vector products can be obtained by embedding the BTTB matrix inside a larger block-circulant with circulant-blocks (BCCB) matrix, which can be diagonalised by the discrete Fourier transform (for details see e.g.\ \cite{AmGr:86,Bu:85}). As a result, by proceeding block-wise, and using the FFT-based approach on each sub-block\footnote{For completeness we note that the $(1,2)$, $(1,3)$, $(2,1)$ and $(3,1)$ sub-blocks of $\tilde\bA$ are rectangular, and must first be embedded inside a square BTTB matrix (with square Toeplitz blocks) by extending the appropriate diagonals and padding by zeros where necessary, before the BCCB-FFT procedure can be applied.}, one can compute matrix-vector products for $\tilde{\bA}$ with a cost $O(\tilde N_j\log{\tilde N_j})$. 







\subsection{Numerical results}
\label{sec:NumericalResults}
We now present results, obtained with the implementation of \S\ref{sec:BEMimplementation}, for two specific examples. 

We first report results for the Koch snowflake. In Figure \ref{fig:KochResults}(a) and Figure \ref{fig:KochResults}(b) we plot the Dirichlet and Neumann components respectively of the BEM solution for plane wave scattering by the fifth order prefractal $\Gamma_5$ (as defined in \S\ref{sec:ClassicalSnowflakes}) with constant impedance parameters $\lambda_+=1.5k(1+\ri)$ and $\lambda_-=k(1+i)$, incident direction $d=(1,1,-1)/\sqrt{3}$, wavenumber $k=20$ and mesh width $h_5=3^{-5}$. The number of degrees of freedom on $\Gamma_5$ and on its containing parallelogram are $N_5=139261$ and $\tilde{N}_5=314281$ respectively. In Figure \ref{fig:KochResults}(c) we show the resulting total field on three faces of a cube of side length $1.4$ centred on the scatterer. In Figure \ref{fig:KochResults}(d) we present a plot of the relative $L^\infty$ error in the scattered field over this cube, for prefractal levels $j=1,\ldots, 4$, using $j=5$ as a reference solution, for $k\in\{20,10,5\}$.  Here the mesh width $h_j=3^{-j}$, so in the limit $j\to \infty$ we expect BEM convergence by Proposition \ref{prop:Koch}. The results in Figure \ref{fig:KochResults}(d) are consistent with this, with the errors decreasing approximately exponentially with increasing $j$. As is to be expected, for fixed $j$ the errors increase as $k$ increases. 

In Figure \ref{fig:SquareResults} we report analogous results for the square snowflake. Figure \ref{fig:SquareResults}(a) and Figure \ref{fig:SquareResults}(b) show the BEM solution on the fourth order prefractal $\Gamma_4$ (as defined in \S\ref{sec:SquareSnowflake}) with the same wavenumber, incident wave and impedance parameters as in the corresponding plots in Figure \ref{fig:KochResults}. The mesh width on each $\Gamma_j$ is now $\sqrt{2}\times 4^{-j}=4^{-j+1/4}$, and the domain plot in Figure \ref{fig:SquareResults}(c) and the relative $L^\infty$ errors in Figure \ref{fig:SquareResults}(d) are computed on a cube of side length 2 centred on the scatterer. 
The number of degrees of freedom on $\Gamma_4$ and on its containing parallelogram are $N_4=188417$ and $\tilde{N}_4=543577$ respectively.
Again, the convergence results in Figure \ref{fig:SquareResults}(d) are consistent with the theoretical predictions of Proposition \ref{prop:Square}.

\begin{figure}[p]
	\subfloat[][
{${\rm Re}([u])$}
]
	{\includegraphics[width=.45\linewidth]{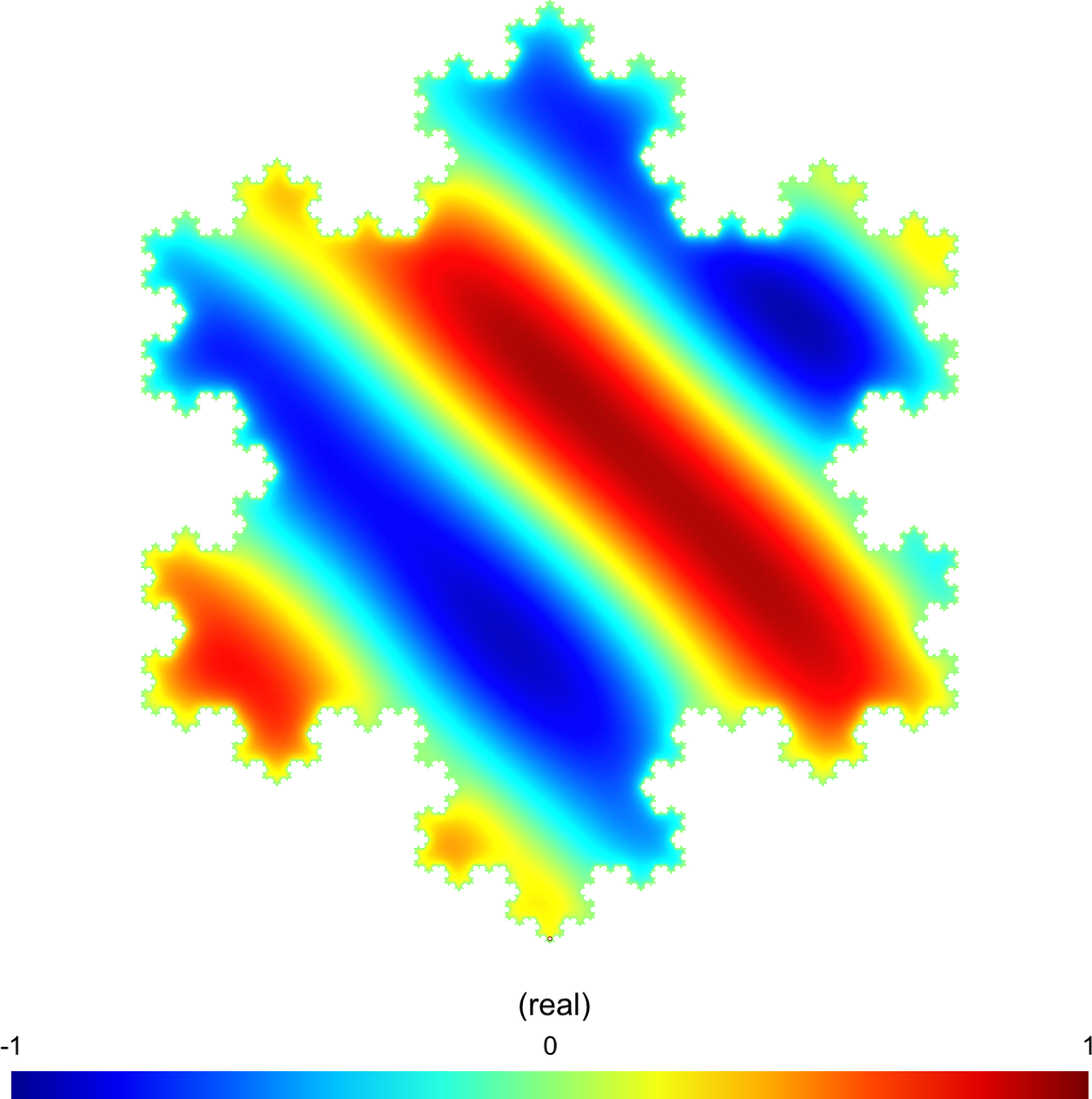}\label{ss11}}
	\hspace{5mm}
	\subfloat[][
{${\rm Re}([\partial_n u])$}
]
		{\includegraphics[width=.45\linewidth]{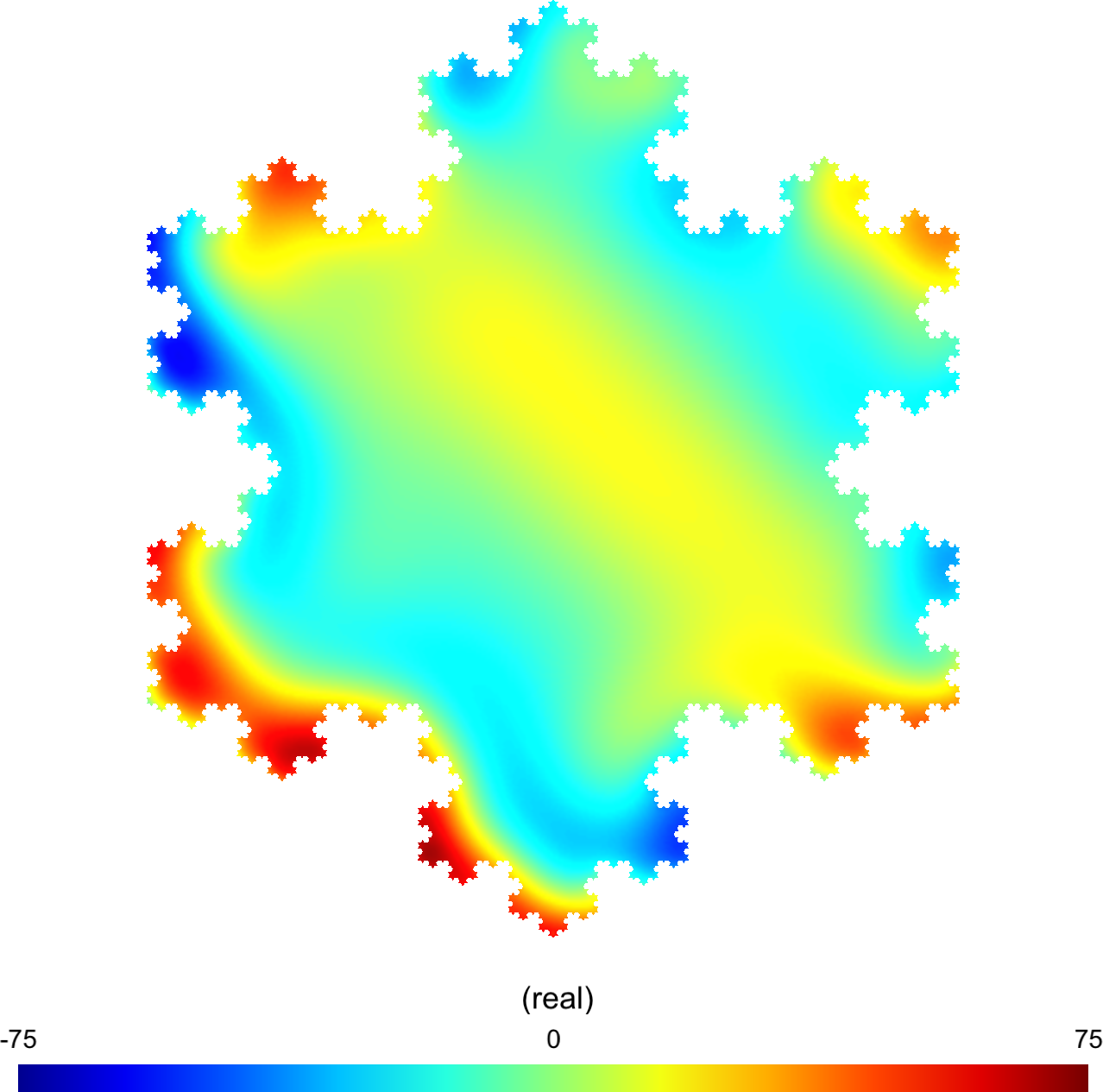}\label{ss12}}\\
	
	\subfloat[][
{${\rm Re}(u+u^i)$}
]
	{\includegraphics[width=.45\linewidth]{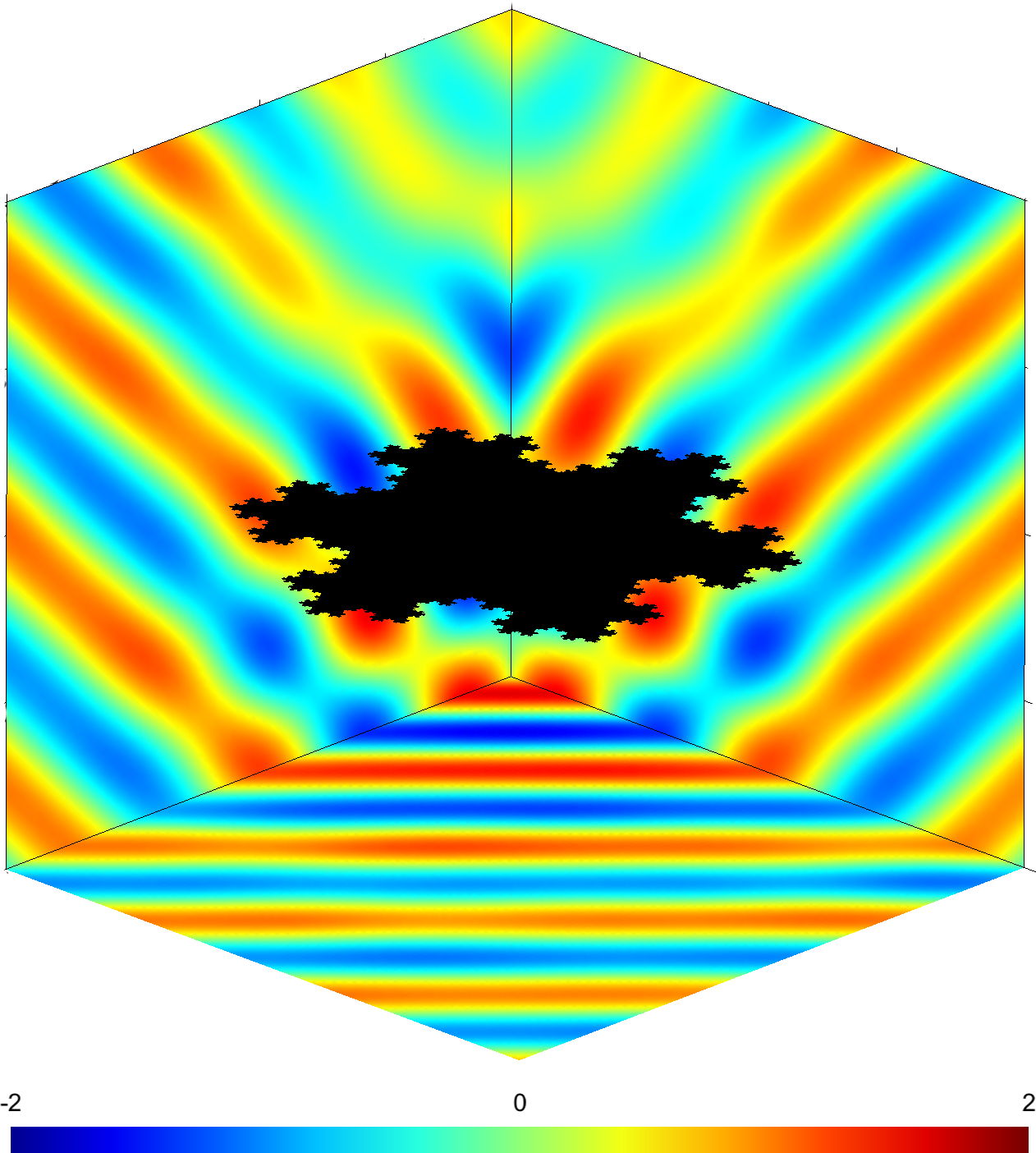}\label{ss21}}
	\hspace{5mm}
	\subfloat[][
{Convergence of $u$}
]
		{\includegraphics[width=.45\linewidth]{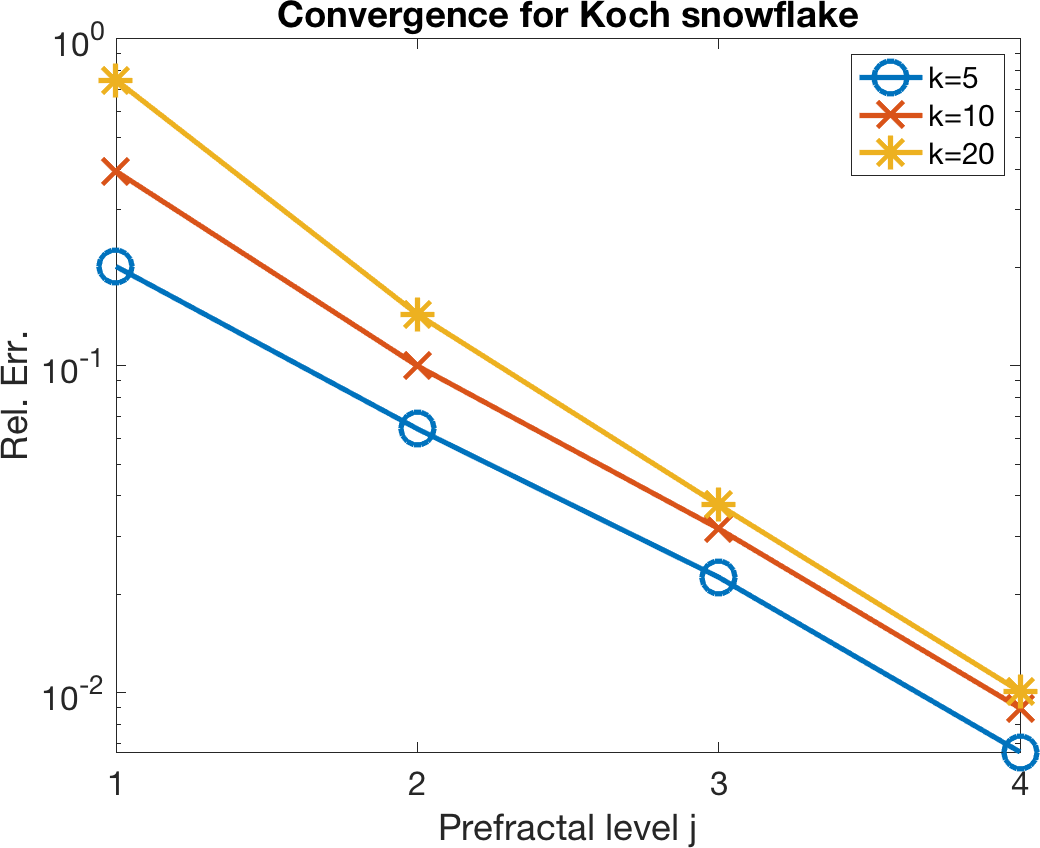}\label{ss22}}
		\caption{Scattering by prefractals of the Koch snowflake. Plots (a) and (b) show respectively the computed Dirichlet and Neumann data on the level $5$ prefractal $\Gamma_5$, for wavenumber $k=20$, plane wave incident direction $d=(1,1,-1)/\sqrt{3}$, impedance parameters $\lambda^+=1.5k(1+\ri)$ and $\lambda^-=k(1+i)$, and mesh width $h_5=3^{-5}$. 
Plot (c) shows the resulting total field $u+u^i$ on three sides of a cube surrounding the scatterer. Plot (d) shows the convergence of the scattered field (relative $L^\infty$ error over the faces of the cube shown in plot (c)) with increasing $j$, 
for three values of $k$, using $j=5$ as reference solution and with mesh width $h_j=3^{-j}$.
	}\label{fig:KochResults}
\end{figure}

\begin{figure}[p]
	\subfloat[][
	{${\rm Re}([u])$}
	]
	{\includegraphics[width=.45\linewidth]{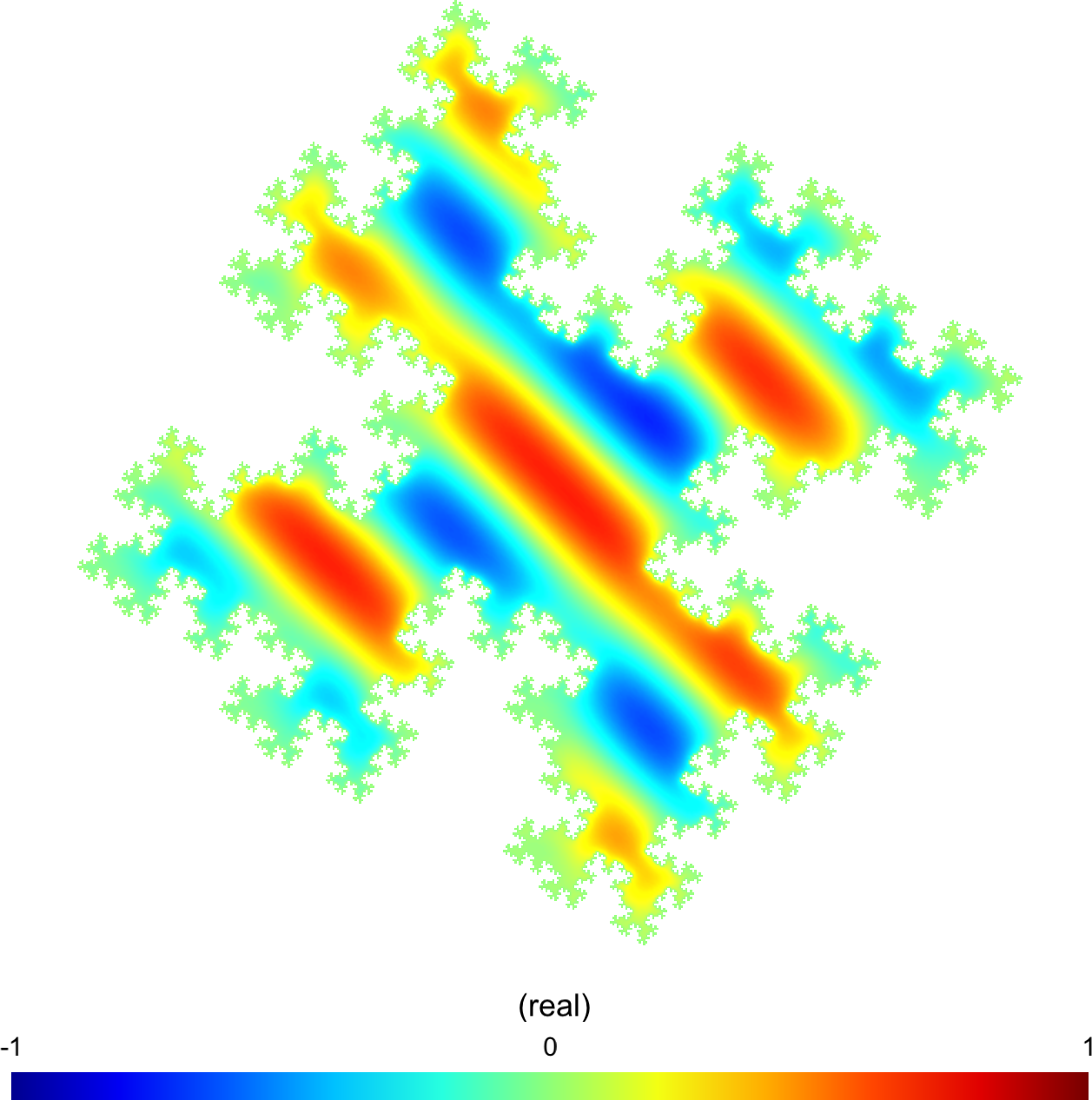}\label{sq11}}
	\hspace{5mm}
	\subfloat[][
	{${\rm Re}([\partial_n u])$}
	]
	{\includegraphics[width=.45\linewidth]{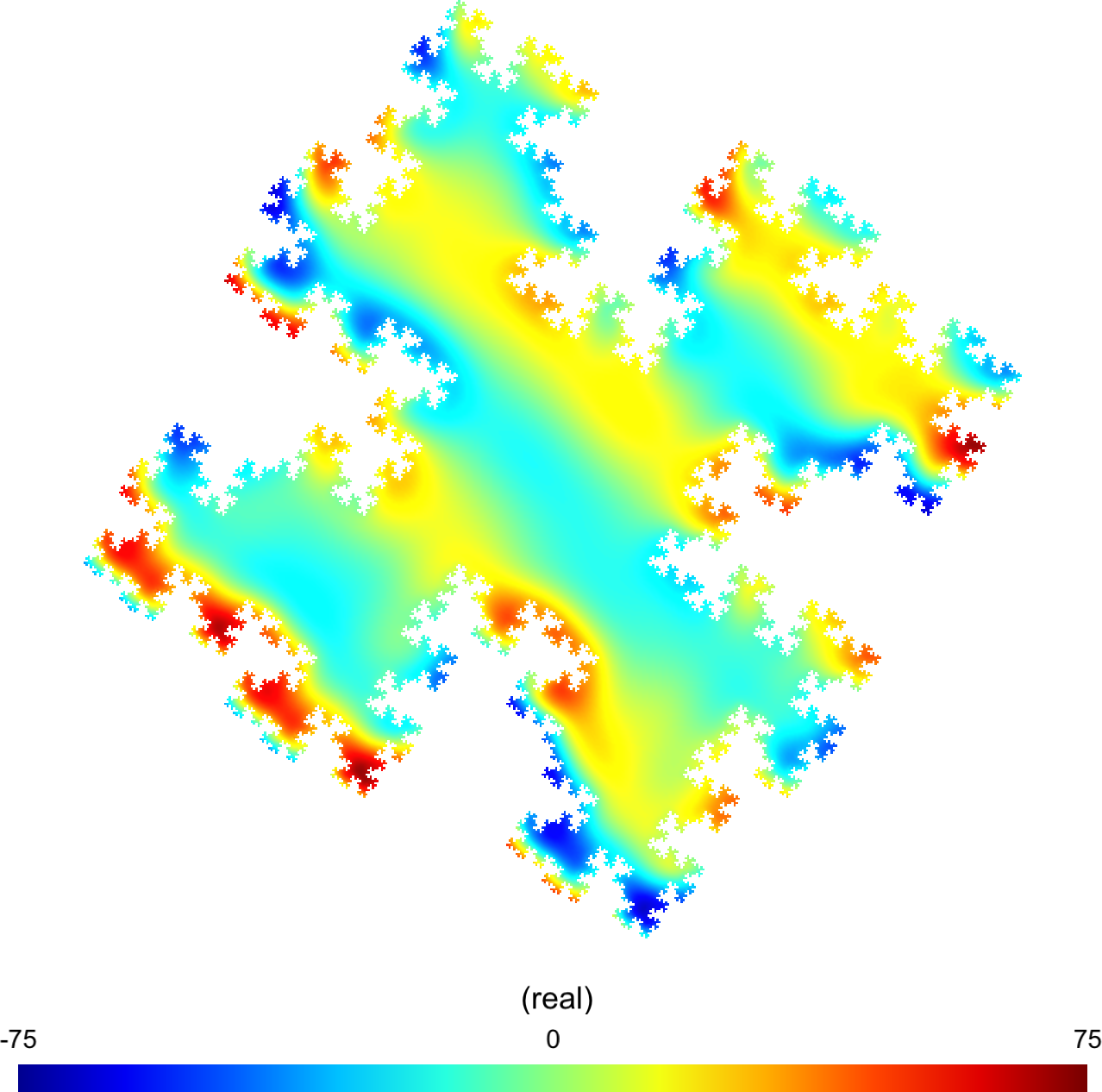}\label{sq12}}\\
	
	\subfloat[][
	{${\rm Re}(u+u^i)$}
	]
	{\includegraphics[width=.45\linewidth]{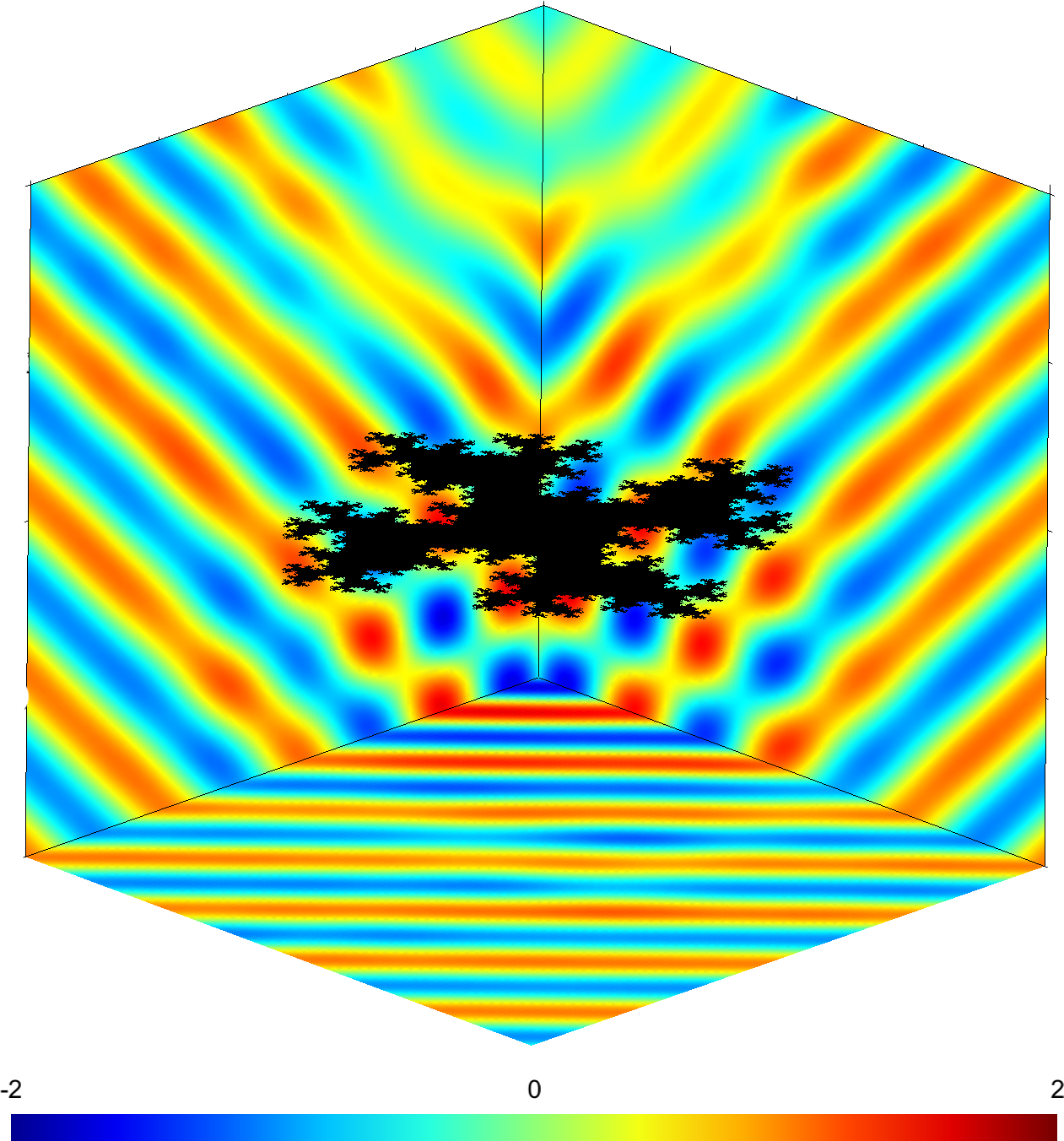}\label{sq21}}
	\hspace{5mm}
	\subfloat[][
	{Convergence of $u$}
	]
	{\includegraphics[width=.45\linewidth]{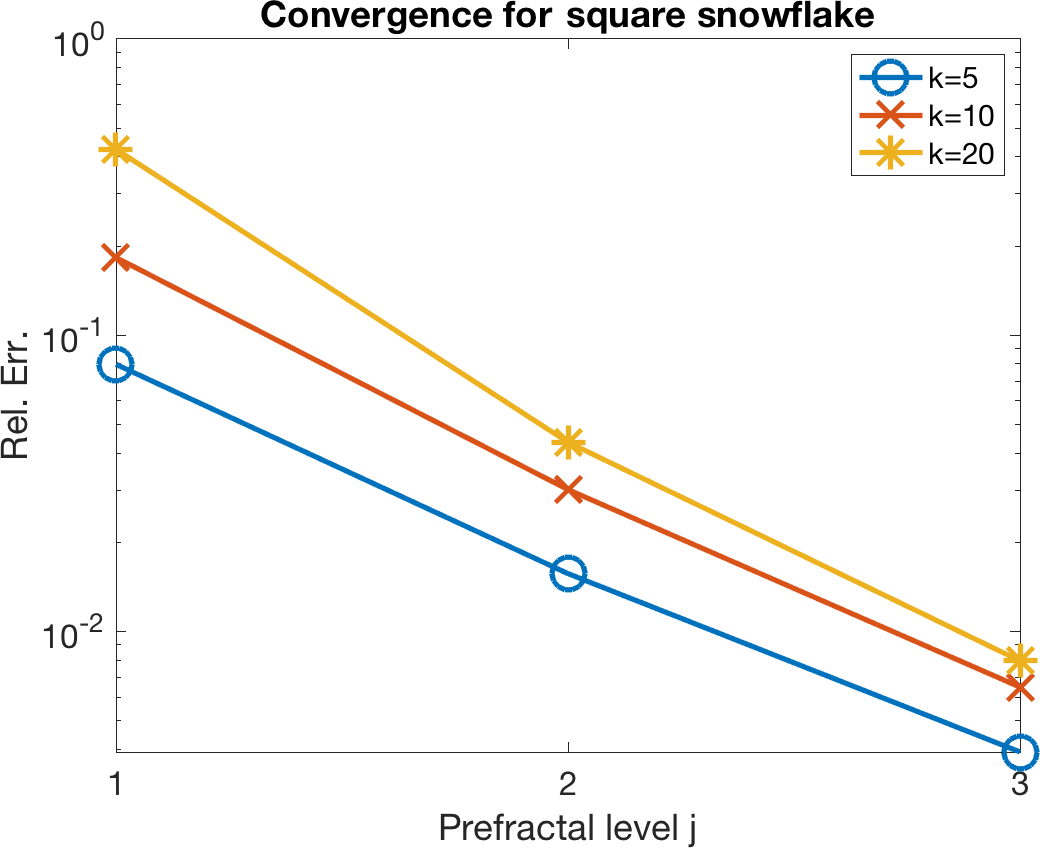}\label{sq22}}
	\caption{Analogue of Figure \ref{fig:KochResults} for the square snowflake. Plots (a)-(c) show results for the level 4 prefractal $\Gamma_4$, with mesh width $h_4=4^{-4+1/4}$. In plot (d) the reference solution is $j=4$, and the mesh width is $h_j=4^{-j+1/4}$.
	}\label{fig:SquareResults}
\end{figure}

\subsection{Scattering by a Cantor dust}
\label{subsec:CantorDustNumerics}
We end the paper by presenting numerical results relating to 
Remark \ref{rem:CantorDust}, which concerns the case where the screen $\Gamma$ is a compact subset of $\Gamma_\infty$ with zero $(n-1)$-dimensional Lebesgue measure. The particular example we consider is where $\Gamma$ is the middle third Cantor dust. Figure \ref{fig:CantorDustNumerics}(a) shows the first six members $\Gamma_0,\ldots,\Gamma_5$ of the standard sequence of prefractal approximations to $\Gamma$, with $\Gamma_0$ the unit square and $\Gamma_j$ formed by removing the ``middle third cross'' from each component of $\Gamma_{j-1}$. Figure \ref{fig:CantorDustNumerics}(b) shows a measure of the magnitude of the corresponding scattered fields (precisely, the $L^\infty$ norm of the scattered field on the boundary of the cube of side length 1.4 centred on each $\Gamma_j$) for wavenumbers $k=5$ and $k=10$, using the same incident wave and impedance parameters as in \S\ref{sec:NumericalResults}.
For this example there is no significant gain to be made in applying the FFT-based solver of \S\ref{sec:BEMimplementation}, because the ratio of the area of $\Gamma_j$ to the area of the smallest parallelogram containing $\Gamma_j$ (i.e.\ the square $\Gamma_0$) decays exponentially (precisely, as $(4/9)^{j}$) with increasing $j$.  Instead, for these experiments each $\Gamma_j$ was meshed with a quasiuniform mesh generated automatically by Bempp, with 26 elements per component of $\Gamma_j$, giving a total of $34\times 4^j$ degrees of freedom on $\Gamma_j$, which equates to 34 and 34816 in the cases $j=0$ and $j=5$ respectively. 
The solutions for $j=0$ and $j=1$ are not expected to be very accurate because for small $j$ we are not using enough elements per wavelength to resolve the solution oscillations. But this is unimportant here as our focus is on the large $j$ behaviour, and for $j\geq 2$ we always have at least 15 elements per wavelength for the largest wavenumber $k=10$.   
According to the theory in Remark \ref{rem:CantorDust}, the magnitude of the scattered field for $\Gamma_j$ should tend to zero as $j\to\infty$, and the numerical results in Figure \ref{fig:CantorDustNumerics}(b) support this, with the $L^\infty$ norm of $u^s$ decaying approximately exponentially with increasing $j$, after an initial preasymptotic phase. 
More precisely, our results are consistent with decay like  $O((4/9)^{j})$ for both wavenumbers, which, as already noted, is the rate at which the area of $\Gamma_j$ tends to zero as $j\to\infty$.

\section{Acknowledgements}
AG and DPH acknowledge support from EPSRC grant EP/S01375X/1, and the authors thank Simon Chandler-Wilde, Timo Betcke and the anonymous reviewer for helpful discussions.

\clearpage

\begin{figure}[t!]
\centering
\subfloat[\label{subfig:CantorDustA}]{
\begin{tikzpicture}
\draw node at (0,0) {\includegraphics[height=60mm]{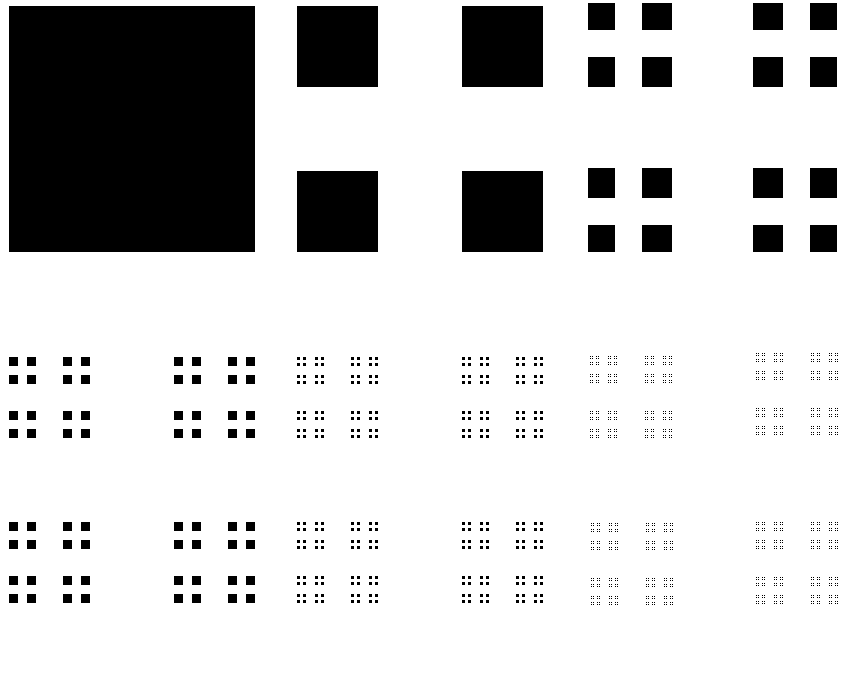}};
\draw node at (-2.5,0.5) {$\Gamma_0$};
\draw node at (0,0.5) {$\Gamma_1$};
\draw node at (2.5,0.5) {$\Gamma_2$};
\draw node at (-2.5,-2.6) {$\Gamma_3$};
\draw node at (0,-2.6) {$\Gamma_4$};
\draw node at (2.5,-2.6) {$\Gamma_5$};
\end{tikzpicture}      
}
\hspace{10mm}
\subfloat[\label{subfig:CantorDustB}]{
\includegraphics[height=65mm]{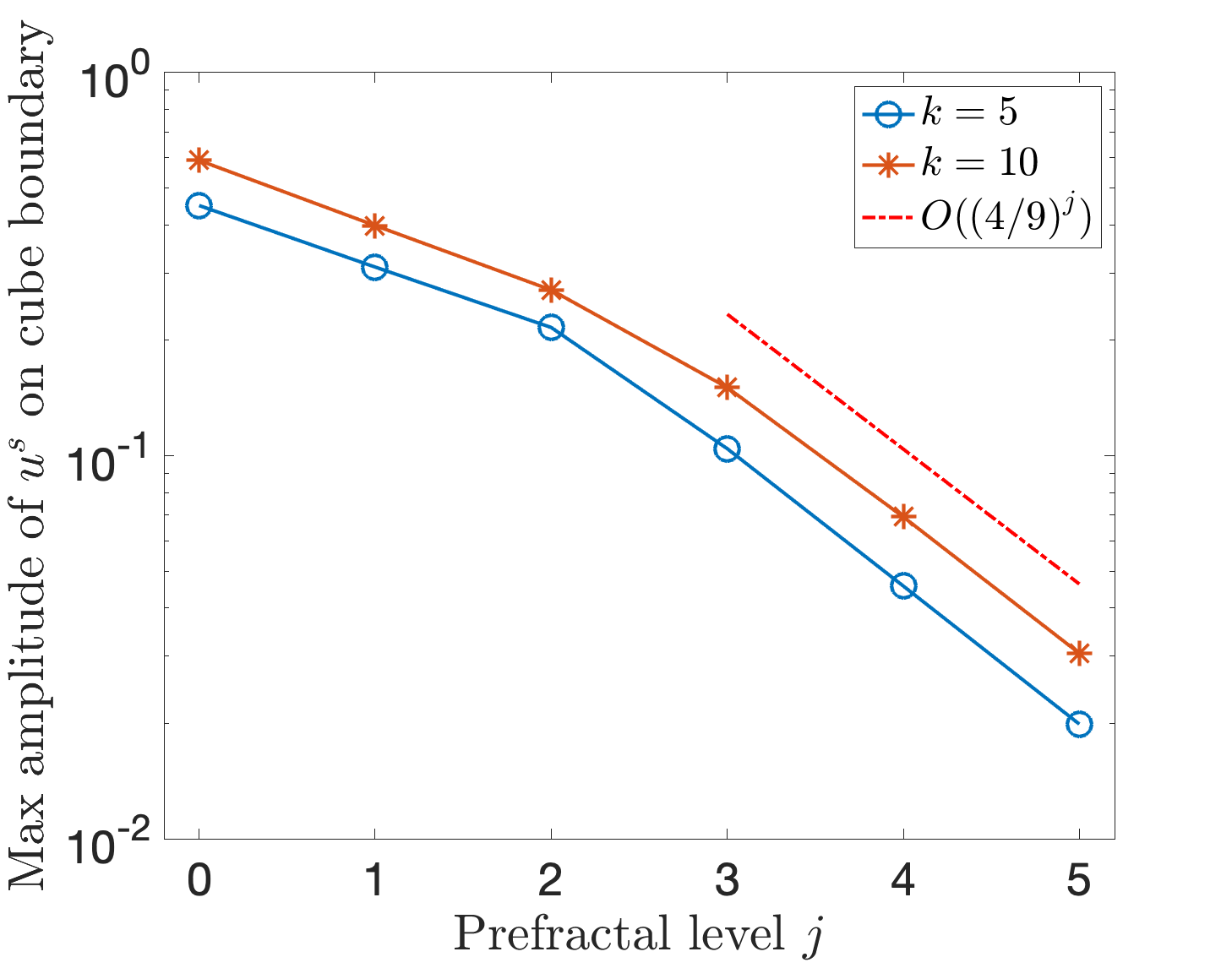}
}
\caption{Scattering by the middle third Cantor dust; (a) shows the prefractals $\Gamma_0, \ldots,\Gamma_5$ and (b) shows the maximum magnitudes of the corresponding scattered fields on the boundary of a cube of side length 1.4 centred on the screen, computed using our BEM.}
\label{fig:CantorDustNumerics}
\end{figure}


\end{document}